\pdfoutput=1

\documentclass[a4paper,11pt]{extarticle}

\usepackage[utf8]{inputenc}
\usepackage[margin=1in]{geometry}
\usepackage{setspace}
\usepackage{graphicx}
\usepackage{float}
\usepackage{subcaption}
\usepackage{caption}

\usepackage{amsmath, amssymb, amsfonts}
\usepackage{mathtools}
\usepackage{amsthm}
\usepackage{dsfont}
\usepackage{gensymb}

\usepackage{algorithmic}
\usepackage{listings}

\usepackage{xcolor}

\usepackage{cite}
\usepackage{url}

\usepackage[hidelinks,hypertexnames=false]{hyperref}

\usepackage{authblk}

\newcommand\numberthis{\addtocounter{equation}{1}\tag{\theequation}}

\newtheorem{theorem}{Theorem}
\newtheorem{lemma}{Lemma}
\newtheorem{remark}{Remark}
\newtheorem{definition}{Definition}
\newtheorem{proposition}{Proposition}

\title{Separation is Optimal for LQR under Intermittent Feedback}

\author{Abdullah Yasin Etcibasi}
\author{C. Emre Koksal}
\author{Eylem Ekici}

\affil{The Ohio State University}

\affil{\texttt{\{etcibasi.1, koksal.2, ekici.2\}@osu.edu}}

\date{}

\begin{document}

\maketitle

\begin{abstract}
	We study finite-horizon linear-quadratic regulation of a scalar linear system with intermittent state feedback under an average communication-rate constraint. In this setting, the scheduling policy and controller are generally coupled through the dual effect: transmission decisions shape future estimation errors, while control actions influence the information available for scheduling. Existing treatments often recover tractability by restricting attention to symmetric scheduling policies, but the optimality of this restriction has remained unclear. We show that, for i.i.d. zero-mean disturbances, symmetric policies are optimal. Consequently, the communication-constrained LQR problem admits a separation structure. The optimal controller is a linear feedback law independent of the scheduling policy, while the optimal scheduler is obtained from a dynamic program. We further show that the optimal scheduling rule is a symmetric threshold policy in the accumulated disturbance since the most recent update.
\end{abstract}

\section{Introduction}

Over the past three decades, wireless communication has advanced rapidly. Starting with early generations, especially 4G, wireless connectivity became an essential part of daily life. With the advent of 5G and the forthcoming 6G, we are now entering a new “robotic era,” characterized by pervasive autonomous agents such as robot vacuums, self-driving vehicles, and aerial drones. Unlike previous eras, these agents must not only communicate with humans but also coordinate among themselves, forming swarms of drones and platoons of vehicles. Consequently, engineers must jointly address communication and control, rather than treat these domains independently.

Traditionally, wireless system designers have tended to ignore control loop requirements, while control engineers have assumed perfect communication channels \cite{park2017wireless}. In practice, control strategies must tolerate network imperfections, and communication protocols should respect the dynamics of the underlying physical systems \cite{hespanha2007survey}. This interplay has given rise to the field of networked control systems (NCSs) \cite{zhang2001stability,walsh2001scheduling}.

As the name suggests, NCS research investigates control loops in which one or more feedback paths traverse a communication network. Early contributions introduced packet losses and delays into feedback control and examined the resulting performance and stability. For example, Åström and Bernhardsson compared the mean variance of a continuous‑time system under periodic sampling and state‑dependent (event‑triggered) sampling, assuming impulsive control actions at sampling instants \cite{astrom2002comparison}. Tabuada proposed an event‑triggered feedback policy for linear time‑invariant (LTI) systems, where the maximum inter-update period depends on the estimation error; the controller employs a zero‑order hold (ZOH), updating only at sampling times and holding the control value constant in between \cite{tabuada2007event}. Montestruque and Antsaklis investigated the stochastic stability of NCSs with time‑varying update intervals (independent and identically distributed (i.i.d.) or Markov‑chain driven) modeling the sampled system as a jump‑linear system \cite{montestruque2004stability}. Heemels et al.\ introduced self‑triggered sampling, where each update time is computed at the preceding update and compared with event‑triggered policies \cite{heemels2012introduction}. Zhang et al.\ separately examined the effects of delays and packet dropouts on stability. Delays were handled by employing system estimates in both full‑state and partial‑state feedback, while packet dropouts were analyzed under periodic transmissions to determine the minimum stable transmission rate \cite{zhang2001stability}. Finally, Lyapunov-based analyses by Heemels and colleagues established the maximum allowable transmission interval (MATI) and maximum allowable delay (MAD) that ensure NCS stability \cite{heemels2010networked}.

Fundamentally, any NCS consists of two key design components. These are the scheduling\footnote{In the literature, this mechanism is referred to by various terms such as scheduling, sampling, transmission, or switching. Throughout this paper, we adopt the term \emph{scheduling}.} mechanism and the controller (see Fig.~\ref{fig:Sys_Model}). In this work, we focus on their joint optimization. We model system uncertainties as i.i.d.\ zero-mean disturbances with a symmetric distribution and employ a standard quadratic cost on state deviations and control effort. This formulation is known as the Linear Quadratic Regulator (LQR), since the objective is to regulate the system state to the origin. This criterion is widely used in the NCS literature~\cite{park2017wireless}.

Our aim is to find the optimal scheduling and control policies for the finite-horizon LQR problem under an average communication-rate constraint. The main difficulty arises from the dual effect, whereby transmission decisions influence future estimation errors while control actions affect the evolution of the estimator. This coupling between the scheduling mechanism and the controller breaks the classical separation principle, which ordinarily allows the estimator and controller to be designed independently without loss of optimality~\cite{bar2003dual}. To regain separability, much of the existing literature imposes simplifying assumptions, most commonly by restricting attention to a subclass of policies under which the conditional mean of the accumulated disturbance is zero. We define such policies as \emph{symmetric policies}. While this assumption renders the problem analytically tractable, it restricts the policy space, raising the question of whether this subclass is optimal. Crucially, however, the optimality of this restricted policy space has yet to be proven, representing a significant gap in the current understanding of communication-constrained LQR problems. Next, we review how this joint optimization problem has been treated in the literature, highlighting the widespread use of the symmetry assumption, sometimes imposed explicitly by assuming zero-mean disturbances and sometimes adopted implicitly through Bernoulli sampling models, or symmetric threshold rules.

In~\cite{molin2009lqg}, the communication-rate constraint was relaxed through a Lagrangian formulation, and separation was enforced by assuming zero-mean disturbances. The design was then recast as an equivalent optimal-estimation problem to enable numerical solutions. An equivalent form of the symmetry assumption was adopted implicitly in~\cite{champati2019performance}, where the design was reformulated as an age-of-information (AoI) minimization problem and a heuristic scheduler was proposed. In a two-hop multi-system setting~\cite{ayan2019age}, disturbances were likewise constrained to be zero-mean, and AoI- and value-of-information (VoI)-based schedulers were compared numerically. Building on this line of work, a VoI metric was introduced in~\cite{wang2021value}, for which a threshold policy was proposed, and the concept was further refined in~\cite{soleymani2021value} with corresponding sampling thresholds.

An alternative line of work was pursued in~\cite{schenato2007foundations}, where sampler updates were modeled as i.i.d. Bernoulli events. It was shown that a linear state-feedback law, with gain given by the Riccati solution, remains optimal under random packet drops, and the resulting stability was analyzed. Similarly, link failures were treated as Bernoulli losses in~\cite{imer2006optimal}, where an impulsive controller (no control when the link fails) was implemented, again recovering the linear feedback controller as LQG-optimal.

Other studies have focused on the joint scheduler and controller design under symmetric threshold policies. In~\cite{ramesh2013design}, a threshold was imposed on the squared estimation error, under which the linear feedback law was found to be optimal. The relation between this threshold and the Lagrange multiplier was investigated in~\cite{molin2014price}, with convergence and stability properties analyzed. Energy constraints at the sensor were incorporated in~\cite{leong2017event}, where a threshold based on the conditional covariance of the estimation error was proposed. More recent work~\cite{aggarwal2025interq} employed a Q-learning–based reinforcement-learning algorithm under the same symmetric assumption, while varying delays were addressed in~\cite{maity2018optimal} through a mixed-integer control problem solved offline with Riccati methods and integer programming. From a differential-game perspective~\cite{aggarwal2024linear}, continuous and intermittent sensing players were contrasted, leading to AoI-based threshold rules. Finally, the disturbance-driven estimation dynamics were correctly modeled in~\cite{maity2020minimal}, where a relaxed dynamic program for the sampler was formulated and an approximate solution proposed due to computational complexity.

In this work, we make the following contributions:
\begin{itemize}
    \item We revisit the constrained joint design of scheduling and control policies for scalar LTI systems, aiming to clarify subtle underlying assumptions.
    
    \item Most importantly, we prove that symmetric policies are optimal for the finite-horizon LQR problem under a communication-rate constraint, thereby establishing that the separation principle holds.
    
    \item Finally, we solve the associated dynamic programming (DP) problem and derive the optimal scheduling and control policies, which take the form of a symmetric threshold rule on the estimation error and a discounted linear feedback law that is independent of the scheduling policy, respectively.
\end{itemize}

We begin by analyzing the information structure of the system and identifying when the classical separation principle breaks down. We then discuss why at least a unit-delay must be introduced between the scheduling mechanism and controller. Then, we introduce the class of symmetric policies and demonstrate that under these policies, the dual effect is neutralized, allowing the estimation and control tasks to decouple. We rigorously prove that this subclass is not only tractable but optimal for the finite-horizon LQR problem under a communication-rate constraint. Finally, we derive the closed-form solution for the optimal system, which consists of a linear feedback control law and a symmetric threshold scheduling policy based on the accumulated disturbances.

\section{System Model}
\label{sec:Sys_Mdl}

\begin{figure}        
  \centering                
  \includegraphics[width=0.6\textwidth]{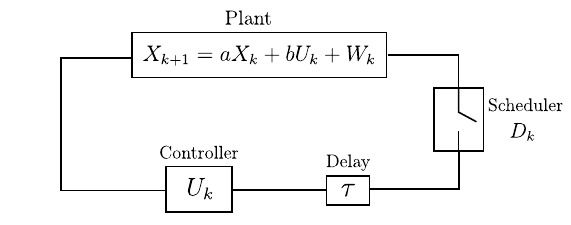}
  \caption{Block diagram of the closed-loop control system}
  \label{fig:Sys_Model}      
\end{figure}

Consider the single-loop, discrete-time, scalar LTI system illustrated in Fig.~\ref{fig:Sys_Model}. 
The plant evolves according to
\begin{equation}
    X_{k+1} = a X_k + b U_k + W_k,
    \label{eqn:system_Model}
\end{equation}
where $X_k$ denotes the system state, $U_k$ the control input, and $W_k$ an additive disturbance. 
The disturbance sequence $\{W_k\}$ is assumed to be i.i.d., zero-mean, with finite variance $\sigma_W^2$, and symmetrically distributed about zero (not necessarily Gaussian). The controller receives intermittent information through the scheduling mechanism with a fixed \(\tau \geq 1\) delay, so that the observation available to the controller is
\begin{equation}
    Y_{k} = D_{k - \tau}\,X_{k - \tau},
    \label{eqn:Y_k}
\end{equation}
where \(D_{k} \in \{0,1\}\) denotes the scheduling decision at time \(k\). In the sequel, we discuss why at least a one-step delay must be introduced between the scheduling mechanism and controller.

Let \(t_k\) denote the most recent update time prior to \(k\), defined as
\[
t_k := \max \{ j \le k \mid D_j = 1 \}.
\]
The controller’s information at time \(k\) is given by
\begin{equation}
    I^\mathcal{C}_{k}
    \;=\;
    \{\,Y_{0},\,U_{0},\,Y_{1},\,U_{1},\,\dots,\,Y_{k-1},U_{\,k-1},\,Y_{\,k}\,\}.
    \label{eqn:Control_info}
\end{equation}
The information vector at the scheduler plays a key role in the optimal scheduling problem. Various approaches and information sets appear in the literature. In the most general form, we denote the scheduler information by
\begin{equation}
    I^\mathcal{S}_{k}
    \;=\;
    \{\,X_{0},\,D_{0},\,U_{0},\,X_{1},\,D_{1},\,U_{1},\,\dots,\,X_{\,k-1},\,D_{\,k-1},U_{k-1},\,X_{\,k}\,\}.
    \label{eqn:update_info}
\end{equation}

\paragraph*{\textbf{Causality of the Information Structure:}}
A policy is said to be causal if, at time $k$, the control input $U_k$ depends only on the information available up to time $k$, i.e., $U_k$ must be measurable with respect to the controller information set $I_k^\mathcal{C}$ defined in~\eqref{eqn:Control_info}. In delay-free formulations, the controller information set $I_k^\mathcal{C}$ is often taken to include the measurement $Y_k = D_k X_k$. However, from~\eqref{eqn:update_info}, the scheduling decision $D_k$ is generated by the scheduler at time $k$ based on $I_k^\mathcal{S}$, which includes $X_k$ but not $D_k$ itself prior to its realization. Since the controller is located downstream of the scheduler, the value of $D_k$ becomes available to the controller only after the scheduler makes its decision. Therefore, if $U_k$ is required to be generated at time $k$ based on $I_k^\mathcal{C}$, including $Y_k = D_k X_k$ in $I_k^\mathcal{C}$ implicitly assumes that the controller has instantaneous access to $D_k$. This creates a circular dependence, as $U_k$ would depend on a variable that is generated at the same time instant and not yet available at the moment of decision, violating causality in the sense of~\cite{witsenhausen1971information}.

To ensure a well-defined causal information structure, the controller must rely only on variables available strictly prior to or at the decision time. This can be enforced by introducing at least a one-step delay between the scheduling mechanism and controller, so that the controller uses $D_{k-1}$ instead of $D_k$ when computing $U_k$. Note that several works in the literature assume delay-free feedback and study separation or optimal scheduling in the communication-constrained LQR problem~\cite{molin2009lqg, molin2014price, molin2011optimal, ramesh2013design, ayan2019age, ramesh2011dual, leong2017event, klugel2020joint}. As discussed, such formulations implicitly rely on information patterns that violates causality, which can affect the validity of the associated dynamic programming formulations. 

With the causal information structure now specified, we define the control and scheduling policies as
\begin{align}
  U_{k} \; &= \; \gamma_{k}\bigl(I^\mathcal{C}_{k}\bigr) \;\in\; \Gamma,\\
  D_{k} \; &= \; f_{k}\bigl(I^\mathcal{S}_{k}\bigr) \;\in\; \Pi,
\end{align}
where the sets \(\Gamma\) and \(\Pi\) contain all admissible control and scheduling decision rules that satisfy the usual measurability and integrability conditions.  
We write  
\(\boldsymbol{\mu}=(\gamma_{0},\gamma_{1},\dots)\in\Gamma\) and  
\(\boldsymbol{\pi}=(f_{0},f_{1},\dots)\in\Pi\)  
for generic control- and scheduling-policy sequences, respectively.  
Because the two sequences are designed jointly, \(\boldsymbol{\mu}\) and \(\boldsymbol{\pi}\) can be coupled.

The main objective is to minimize constrained expected LQR cost over a finite horizon
\begin{align}
  P_1 : & \min_{\boldsymbol{\pi}\in\Pi,\;\boldsymbol{\mu}\in\Gamma}   \frac{1}{N}\mathbb{E}\left[\sum_{k=0}^{N-1}\bigl(q\,X_k^2 + r\,U_k^2 \bigr) + qX_N^2\right] \label{eqn:Original_Objective} \\
  & s.t. \quad \: \frac{1}{N} \sum_{k=0}^{N-1} D_k \leq r_{s}
\end{align}
where \(q,\,r,\,r_{s} > 0\).

\begin{theorem} 
\label{theorem:optimal control}
    Consider the system \eqref{eqn:system_Model}–\eqref{eqn:update_info} and the optimization problem \(P_{1}\). The optimal controller is given by the linear feedback law 
    \begin{equation}
        U_k = \gamma_k(I^\mathcal{C}_{k}) = -L\,\hat{X}_k,
        \label{eqn:optimal_controller}
    \end{equation}
    where 
    \begin{equation}
        \hat{X}_k = \mathbb{E}\bigl[X_k\mid I^\mathcal{C}_{k}\bigr], 
        \quad
        L = \frac{a b\,P}{r + b^2 P},
        \quad
        P = a^2P + q - \frac{(a b P)^2}{r + b^2 P}.
    \end{equation}
\end{theorem}

\begin{proof}
    See Appendix~\ref{Appendix_A}.
\end{proof}

Let us define the estimation error at the controller as
\begin{equation}
    \mathcal{E}_{k} \;=\; X_{k} \;-\; \hat{X}_{k}.
    \label{eqn:est_error_original}
\end{equation}
Given the optimal controller in~\eqref{eqn:optimal_controller}, the objective function in~\eqref{eqn:Original_Objective} can be reformulated in terms of the estimation error~\(\mathcal{E}_{k}\). The goal is then to determine the optimal scheduling policy that minimizes this reformulated cost.

\begin{theorem} \label{theorem:equivalent problem}
  For a scalar LTI system with \(q,r>0\), a fixed delay $\tau \geq 1$, and using the optimal controller in \eqref{eqn:optimal_controller}, problem \(P_{1}\) can be equivalently written as
  \begin{align}
    P_{2}:\quad 
      &\min_{\boldsymbol{\pi}\; \in \; \Pi}\;  \frac{1}{N}\,\mathbb{E}\left[\sum_{k=0}^{N-1} \mathcal{E}_{k}^{2}\right] 
      \label{eqn:P_2}\\
    \text{s.t.}\quad 
      & \frac{1}{N}\sum_{k=0}^{N-1} D_{k} \;\le\; r_{s}. 
      \nonumber
  \end{align}
\end{theorem}

\begin{proof}
    See Appendix~\ref{Appendix_B}.
\end{proof}
Note that in Theorem~\ref{theorem:optimal control}, we use the steady-state solution of the Riccati equation for notational simplicity, which leads to a simplified form of the equivalent problem $P_2$. In the transient case, the only difference is the appearance of the term $L_k (r + b^2 P_{k+1})$ in the summation. This term does not fundamentally affect the analysis. For completeness, we provide the key equations that change under the transient controller in Appendix~\ref{appendix:E}.

In formulation \(P_2\), the cost depends only on the estimation error. Consequently, if the estimation‐error dynamics can be written independently of the control inputs, the control and scheduling problems decouple, i.e., the separation principle holds. Note that without the separation principle, although Theorem \ref{theorem:optimal control} gives the optimal control law \( U_k = -L\,\hat X_k, \) the state estimate \(\hat X_k\) itself depends on the scheduling policy \(\boldsymbol{\pi}\). Until \(\boldsymbol{\pi}\) is fixed, the estimator remains undefined and the closed‐loop law cannot be completed. Only once \(\boldsymbol{\pi}\) is chosen can both the estimator and the corresponding optimal controller be fully specified.

\section{Symmetric Policies}
In this section, we begin by deriving the estimation error dynamics to identify the precise conditions under which the separation principle remains valid. Within this framework, we formally define the class of symmetric policies, a subclass characterized by the property that the accumulated disturbance remains conditionally zero-mean with respect to the controller's information. While such policies are often treated as a mere modeling convenience, they have silently shaped much of the literature in event-triggered control, frequently appearing as hidden or unstated assumptions.

We demonstrate that the statistical coupling between scheduling decisions and disturbances fundamentally prevents estimation error decoupling unless this symmetry is imposed. Ultimately, this section provides the necessary foundation for Section 5, where we move beyond tractability to rigorously prove that these symmetric policies are, in fact, optimal.

\subsection{Estimation Error Dynamics and the Role of Symmetry}

The analysis begins by deriving the estimation error dynamics based on the controller’s available information, setting the stage for examining the validity of the separation principle.

\begin{lemma} \label{lemma:est_err}
    Consider the system \eqref{eqn:system_Model}–\eqref{eqn:Y_k} under the information structures 
    \eqref{eqn:Control_info}–\eqref{eqn:update_info}. 
    Then the estimation error dynamics are given as follows:  

    \begin{align}
        \mathcal{E}_{k+1} 
        &= (1 - D_{k+1-\tau})\left(
            a\Bigl(X_k - \mathbb{E}[ X_k \mid I^\mathcal{C}_{k}, D_{k+1-\tau}=0 ]\Bigr) + W_k
          \right) \nonumber \\[4pt]
        &\quad + D_{k+1-\tau}\left(
            \sum_{j=0}^{\tau-1} a^{\tau-1-j}\,W_{k+1-\tau+j}
          \right).
        \label{eqn:est_error_wDelay}
    \end{align}
\end{lemma}
\begin{proof}
For $\tau \geq 1$ and the scheduling instant $t_k$ , we obtain
\begin{align*}
    \mathcal{E}_{t_k+\tau} 
    &= a^{\tau}\Bigl(X_{t_k} - \mathbb{E}[X_{t_k} \mid I^\mathcal{C}_{t_k+\tau}]\Bigr) 
       + \sum_{j=0}^{\tau-1} a^{\tau-1-j}\Bigl(W_{t_k+j} - \mathbb{E}[W_{t_k+j}\mid I^\mathcal{C}_{t_k+\tau}]\Bigr) \\[4pt]
    &\overset{(a)}{=} \sum_{j=0}^{\tau-1} a^{\tau-1-j}\Bigl(W_{t_k+j} - \mathbb{E}[W_{t_k+j}\mid I^\mathcal{C}_{t_k+\tau}]\Bigr) \\[4pt]
    &\overset{(b)}{=} \sum_{j=0}^{\tau-1} a^{\tau-1-j} W_{t_k+j}.
\end{align*}
Step (a) holds because $X_{t_k}$ is included in $I^\mathcal{C}_{t_k+\tau}$, so its conditional error vanishes.  
Step (b) follows because, for $j=0,\dots,\tau-1$, the noise variables $W_{t_k+j}$ are zero-mean and independent of the sigma-algebra \(\mathcal{F}_{k,\tau} := \sigma\!\bigl(I^\mathcal{C}_{t_k+\tau}\bigr)\), which contains only information available up to time $t_k$ (including the scheduling decision $D_{t_k}$)  but not the realizations of $W_{t_k}, \ldots, W_{t_k+\tau-1}$. 
Hence, \( \mathbb{E}[\,W_{t_k+j} \mid \mathcal{F}_{k,\tau}] = 0 \).
Finally, for $m \geq 0$, we have
\begin{align*}
    \mathcal{E}_{t_k+\tau+m+1} 
    &= aX_{t_k+\tau+m} + bU_{t_k+\tau+m} + W_{t_k+\tau+m} \\
    &\quad - \mathbb{E}\!\left[aX_{t_k+\tau+m} + bU_{t_k+\tau+m} + W_{t_k+\tau+m} 
        \mid I^\mathcal{C}_{t_k+\tau+m}, \, D_{t_k+m+1}=0\right] \\[4pt]
    &= a\Bigl(X_{t_k+\tau+m} - \mathbb{E}[X_{t_k+\tau+m} \mid I^\mathcal{C}_{t_k+\tau+m}, D_{t_k+m+1}=0]\Bigr) + W_{t_k+\tau+m}.
\end{align*}
Here we used the fact that $U_{t_k+\tau+m}\in I^\mathcal{C}_{t_k+\tau+m+1}$ and that
\[
D_{t_k+m+1} \;\leftarrow\; I^\mathcal{S}_{t_k+m+1} \;\leftarrow\; X_{t_k+m+1} \;\leftarrow\; \{W_{t_k}, \dots, W_{t_k+m}\}.
\]
\end{proof}

We now examine the conditional expectation term appearing in \eqref{eqn:est_error_wDelay}. Due to the following dependency chain:
\begin{equation}
    D_{k+1-\tau} 
    \;\leftarrow\; I^\mathcal{S}_{k+1-\tau} 
    \;\leftarrow\; X_{k+1-\tau} 
    \quad
    \text{and} \quad X_{k} \;\leftarrow\; X_{k+1-\tau}.
    \label{eqn:dep_chain_X}
\end{equation}
there exist admissible scheduling policies under which
\begin{equation}
  \mathbb{E}\!\left[X_k \mid I^\mathcal{C}_{k},\, D_{k+1-\tau}=0 \right]
  \;\neq\;
  \mathbb{E}\!\left[X_k \mid I^\mathcal{C}_{k} \right].
  \label{eqn:cond_Xk_delay}
\end{equation}
In much of the existing literature, this dependence is implicitly neglected by assuming that~\eqref{eqn:cond_Xk_delay} holds with equality, which simplifies the analysis. However, this assumption is not valid in general and requires justification.

\begin{definition}
For all $\tau \geq 1$, we define the subclass of admissible scheduling policies as
\begin{align*}
    \Pi' \triangleq \Big\{ \pi \in \Pi \;\Big|\; 
    &\mathbb{E}\!\left[X_{k} \mid 
       I^\mathcal{C}_{k},\, D_{k+1-\tau}=0\right] = \mathbb{E}\!\left[X_{k} \mid I^\mathcal{C}_{k}\right] \Big\}.
    \numberthis \label{defn:Pi_prime}
\end{align*}
\end{definition}

\noindent This subclass enforces that conditioning on future scheduling decisions does not alter the conditional mean of the state. Additionally, note that for policies that do not belong to \(\Pi'\), the following strict inequality holds:
\begin{equation}
    \forall\, \pi \in \Pi \setminus \Pi':\quad
    X_k - \mathbb{E}\left[X_k \mid I^\mathcal{C}_{k},\, D_{k+1-\tau}=0\right]
    \;\neq\; \mathcal{E}_k.
\end{equation}

We now introduce the core subclass of policies, referred to as \emph{symmetric policies}. 
To this end, we consider the system state evolution. 
Since the system is linear, the state at time \(t_k + m\) can be expressed as
\begin{equation}
    X_{t_k + m} = a^{m} X_{t_k} + \sum_{j=0}^{m-1} a^{m-1-j} \, b \, U_{t_k + j} + \sum_{j=0}^{m-1} a^{m-1-j} \, W_{t_k + j}.
    \label{eqn:X_tk+m}
\end{equation}
This expression follows directly from the linear dynamics and does not depend on the delay parameter~\(\tau\).

We define the accumulated disturbance since the most recent update as
\begin{equation}
    S_m \triangleq \sum_{j=0}^{m-1} a^{m-1-j} \, W_{t_k + j} \label{defn:S_m}
\end{equation}
With this definition, the conditional expectation of the system state given the controller information can be written as
\begin{equation}
    \mathbb{E}\left[X_{t_k+m} \mid I^\mathcal{C}_{t_k+m}\right] = a^{m} X_{t_k} + \sum_{j=0}^{m-1} a^{m-1-j} \, b \, U_{t_k + j} + \mathbb{E}\left[S_m \mid I^\mathcal{C}_{t_k+m} \right] , \quad \forall \; m \geq \tau
    \label{eqn:Exp_X_tk+m}
\end{equation}
This term is critical, as it directly determines the optimal controller in~\eqref{eqn:optimal_controller}.

We now shift our attention to the expectation term in~\eqref{eqn:Exp_X_tk+m}, namely \(\mathbb{E}\!\left[S_m \mid I^\mathcal{C}_{t_k+m}\right]\).
Due to the following dependency chain
\begin{equation}
    I^\mathcal{C}_{t_k+m} \;\leftarrow\; D_{t_k+m-\tau} \;\leftarrow\; I^\mathcal{S}_{t_k+m-\tau} \;\leftarrow\; X_{t_k+m-\tau} \;\leftarrow\; W_{t_k+m-\tau-1}, 
    \label{eqn:dep_chain_Wk}
\end{equation}
there exist admissible scheduling policies under which the disturbance \(W_k\) becomes statistically dependent on the scheduling decision. For instance, if the scheduler employs a threshold-based policy of the form \(D_k = \mathds{1}_{\{W_{k-1} > \beta\}}\) for some \(\beta > 0\), then the event \(D_k\) directly reveals information about the disturbance \(W_{k-1}\), inducing statistical dependence. Consequently, the conditional mean of the disturbance may be nonzero, i.e.,
\begin{equation}
    \mathbb{E}\left[ W_{t_k+j} \mid I^\mathcal{C}_{t_k+m} \right] \;\neq\; 0, \qquad \forall \; m>j
    \label{eqn:cond_Wk}
\end{equation}
and hence,
\begin{equation}
    \mathbb{E}\left[ S_m \mid I^\mathcal{C}_{t_k+m} \right] \;\neq\; 0
\end{equation}

\begin{definition}
    \noindent We define the class of scheduling policies for which the accumulated disturbance remains conditionally independent of the controller’s information over multiple steps as
    \begin{equation}
    \Pi'' = \left\{ \pi \in \Pi \;\middle|\; 
    \mathbb{E}\left[ S_m \mid I^\mathcal{C}_{t_k + m} \right] 
    = 0,\; \forall\, k,\,  m \in \mathbb{N} \right\}, \label{defn:sym_pol}
    \end{equation}
    where $S_m$ is defined in \eqref{defn:S_m} and refer to such policies as \textit{symmetric policies}.
\end{definition}

Intuitively, these are referred to as \emph{symmetric policies} because, under such policies, the conditional distribution of the accumulated disturbance remains symmetric (zero mean) with respect to the controller’s information, irrespective of the scheduling decisions. In contrast, non-symmetric policies may induce a bias in the conditional expectation of the accumulated disturbance.

\begin{proposition} \label{prop:Subset}
    The set of symmetric  policies defined in~\eqref{defn:sym_pol} is a subset of \(\Pi'\); that is, \(\Pi'' \subseteq \Pi'\).
\end{proposition}

\begin{proof}
    Consider the system state evolution~\eqref{eqn:X_tk+m} and its conditional expectation~\eqref{eqn:Exp_X_tk+m}. 
    Accordingly, the estimation error at time \(t_k + \tau + m\) is
    \begin{align*}
        \mathcal{E}_{t_k + \tau + m} 
        &= X_{t_k + \tau + m} - \mathbb{E}\left[ X_{t_k + \tau + m} \mid I^\mathcal{C}_{t_k + \tau + m} \right] \\
        &= a^{\tau+m}\; \left(X_{t_k} - \mathbb{E}[X_{t_k} \mid I^\mathcal{C}_{t_k + \tau + m}] \right) \\
        & \qquad \qquad + \sum_{j=0}^{\tau + m-1} a^{\tau + m-1-j} \left(W_{t_k + j} - \mathbb{E}\left[ W_{t_k + j} \mid I^\mathcal{C}_{t_k + \tau + m} \right] \right) \\
        & \overset{(a)}{=} \sum_{j=0}^{\tau + m-1} a^{\tau + m-1-j} \left(W_{t_k + j} - \mathbb{E}\left[ W_{t_k + j} \mid I^\mathcal{C}_{t_k + \tau + m} \right] \right) \\
        & \overset{(b)}{=} \sum_{j=0}^{\tau + m-1} a^{\tau + m-1-j} W_{t_k + j} - \sum_{j=0}^{m-1} a^{m-1-j} \; \mathbb{E}\left[ W_{t_k + j} \mid I^\mathcal{C}_{t_k + \tau + m} \right],
    \end{align*}
    where the step (a) is because $X_{t_k}$ is measurable with respect to $I^\mathcal{C}_{t_k + \tau + m}$, and the step (b) is because $I^\mathcal{C}_{t_k + \tau + m}$ depends on $D_{t_k+m}$ and $D_{t_k+m}$ depends on $\{W_0,...,W_{t_k+m-1}\}$
    
    Under symmetric policies~\eqref{defn:sym_pol}, we have
    \begin{equation}
        \mathcal{E}_{t_k +\tau + m} = \sum_{j=0}^{\tau + m-1} a^{\tau+m-1-j} \, W_{t_k + j}, \label{eqn:est_error_sym_pol}
    \end{equation}
    since \(\mathbb{E}\left[ S_m \mid I^\mathcal{C}_{t_k + m} \right] = 0\), by definition. 

    As a result, the estimation error dynamics reduce to
    \begin{equation}
        \mathcal{E}_{k+1} =
        (1 - D_{k+1-\tau})\left( a\;\mathcal{E}_{k} + W_k \right) \; + \; D_{k+1-\tau}\left( \sum_{j=0}^{\tau-1} a^{\tau-1-j}\,W_{k+1-\tau+j} \right)
        \label{eqn:est_error_sep}
    \end{equation}
    which corresponds to a scenario where
    \begin{equation}
        \mathbb{E}\!\left[X_{k} \mid 
       I^\mathcal{C}_{k},\, D_{k+1-\tau}=0\right] = \mathbb{E}\!\left[X_{k} \mid I^\mathcal{C}_{k}\right] ,
    \end{equation}
    satisfying the condition in~\eqref{defn:Pi_prime}. Hence, \(\pi \in \Pi'' \Rightarrow \pi \in \Pi'\), and thus \(\Pi'' \subseteq \Pi'.\) \footnote{Equality holds in the noise‑free case \(\sigma_W^2=0\).}
\end{proof}

Hence, under symmetric policies, problem \(P_2\) reduces to the following form:
\begin{align*}
  P_3:\quad 
    & \min_{\pi \in \Pi''} \frac{1}{N}\,\mathbb{E}\left[\sum_{k=0}^{N-1} \mathcal{E}_{k}^{2}\right]\\
    \text{s.t.}\quad 
    & \frac{1}{N}\sum_{k=0}^{N-1} D_{k} \;\le\; r_{s} \nonumber \\
    & \mathcal{E}_{k+1} = 
    (1 - D_{k+1-\tau})\left( a\;\mathcal{E}_{k} + W_k \right) \;+ \; D_{k+1-\tau}\left( \sum_{j=0}^{\tau-1} a^{\tau-1-j}\,W_{k+1-\tau+j} \right). \nonumber
\end{align*}

Observe that \(P_3\) depends only on the estimation error dynamics and is independent of the control input. Consequently, the control and scheduling policies decouple, allowing us to focus exclusively on the scheduling policy and thereby invoke the separation principle. However, since not all scheduling policies satisfy the error evolution in~\eqref{eqn:est_error_sep}, it remains to show that symmetric policies are optimal. Establishing this result renders problems \(P_2\) and \(P_3\) equivalent, in other words, separation holds.

\subsection{The Symmetry Assumption in Prior Work}

In this section, we begin by reviewing common triggering mechanisms in the literature that motivate the symmetric policy assumption in the delay-free setting. As noted in Sec.~\ref{sec:Sys_Mdl}, delay-free feedback breaks causality, consequently, works that assume a delay-free setting implicitly adopt a noncausal information structure. We then turn to studies that incorporate feedback delay and show that they likewise rely (often implicitly) on symmetric policies.

We begin by examining various designs of the scheduler information structure \(I^\mathcal{S}_{k}\) and identifying the precise conditions under which the separation principle holds, thereby yielding problem \(P_3\). A central requirement is the removal of the dependency chain in~\eqref{eqn:dep_chain_Wk}. To this end, the literature has focused on two primary triggering strategies: \emph{event‑triggered} schemes and \emph{self‑triggered} schemes~\cite{aastrom2012introduction, park2017wireless, astrom2002comparison}.

In event-triggered mechanisms, the scheduler continuously monitors the system state (or an error signal) and initiates an update only when necessary. Since scheduling decisions depend directly on the system state, both dependency chains in~\eqref{eqn:dep_chain_X} and~\eqref{eqn:dep_chain_Wk} remain intact, and the separation principle generally does not apply.

In contrast, self-triggered mechanisms determine the next update time in advance (e.g., periodic scheduling), without monitoring the system between scheduling instances. In such cases, the scheduling actions are decoupled from the real-time state, effectively breaking the dependency chain in~\eqref{eqn:dep_chain_Wk}, and thus the separation principle holds~\cite{champati2019performance, antunes2019consistent, gatsis2014optimal}.

In addition to the explicit triggering mechanisms discussed above, there exist alternative system model designs that implicitly enforce a self-triggered structure and, consequently, ensure the validity of the separation principle. For instance, if the scheduler is co-located with the controller (i.e., implemented on the controller side), it does not have direct access to the system state. In this case, the scheduling mechanism is functionally equivalent to a self-triggered scheme~\cite{aggarwal2024linear}. At first glance, this modeling choice may appear to impose an additional restriction on the scheduler and potentially complicate the problem formulation. However, this restriction effectively breaks the dependency chain, under which periodic scheduling becomes optimal, as shown in greater detail in~\cite{etcibasi2026freshness}. We also showed that increasing the communication rate, and thereby decreasing the mean age of information, does not necessarily reduce the LQR cost for every scheduling policy.

Alternatively, the same decoupling of information can arise even when the scheduler is not physically co-located with the controller. Specifically, if the controller does not have knowledge of the scheduling policy—namely, it does not have access to the decision rule \( f_k(I^\mathcal{S}_{k}) \) governing the scheduling actions—then the dependency chain in~\eqref{eqn:dep_chain_Wk} is likewise broken. A similar modeling assumption is adopted in~\cite{aggarwal2025interq}, where this dependency is implicitly removed as well. At this point, we emphasize the following critical observation:

\begin{remark}
    From the controller's perspective, knowing the scheduling decisions \(D_k \in \{0,1\}\) does not necessarily imply knowledge of the underlying scheduling rule \(f_k(I^\mathcal{S}_{k})\). If the controller does know this rule (even though updates are received intermittently), the information flow from the scheduler to the controller becomes uninterrupted. In other words, the absence of a transmission can itself carry information.
    However, this phenomenon is policy-dependent rather than intrinsic to event-triggered communication.
    In particular, under symmetric policies, silence is statistically non-informative and does not affect the controller’s conditional estimate.
    \label{remark:no_tx_is_tx}
\end{remark}

Remark~\ref{remark:no_tx_is_tx} can be interpreted as follows: although the scheduler transmits exact state measurements intermittently, it implicitly conveys information at all times by revealing whether a transmission occurred. In effect, even the absence of a transmission restricts the set of possible values that the state could have taken, thereby reducing uncertainty. Thus, the scheduler is constantly providing information—albeit indirectly. A similar observation is made in~\cite{maity2020minimal}. However, symmetric policies eliminate this information coupling. Consequently, separation holds: the control policy is independent of the scheduling actions, and vice versa.

Event-triggered scheduling schemes generally conflict with the requirements of symmetric policies. While some studies impose this assumption explicitly~\cite{molin2009lqg, ramesh2011dual, ramesh2013design}, most adopt it implicitly, even when their system models do not support it~\cite{molin2014price, molin2011optimal, ayan2019age, leong2017event, klugel2020joint, demirel2018tradeoffs, gatsis2016state, molin2014bi, mamduhi2019try, aggarwal2023weighted, vilgelm2017control, lu2023full}. It is also important to note that these works do not account for delays in their system models, resulting in a noncausal information structure. Furthermore, although some of these works consider the presence of measurement noise, our observations remain valid in this setting.

We now investigate the effect of the delay. In the presence of delay, certain conditional 
expectation terms involving the disturbance vanish for all policies. For example, 
when $\tau = 1$,
\[
    \mathbb{E}[ W_k \mid I^\mathcal{C}_{k+1}] = 0.
\]
At first glance, introducing a delay might appear to resolve the dependency issue and restore the validity of the separation principle.  

However, this vanishing property is not a consequence of any policy choice. This does not imply that $\pi \in \Pi''$ (i.e., that $\pi$ is a symmetric policy). 
In particular, even with delay we cannot conclude that
\[
\mathbb{E}\!\left[X_{k} \mid I^\mathcal{C}_{k},\, D_{k+1-\tau}=0\right] 
= \mathbb{E}\!\left[X_{k} \mid I^\mathcal{C}_{k}\right].
\]
In other words, in the delayed scenario the condition in~\eqref{defn:sym_pol} is not guaranteed to hold for arbitrary policies, and hence Proposition~\ref{prop:Subset} does not apply.

If, on the other hand, one assumes that the estimation error dynamics take the form
\begin{equation}
    \mathcal{E}_{k+1} 
    = (1 - D_{k+1-\tau})\bigl(a\,\mathcal{E}_k + W_k\bigr) 
      + D_{k+1-\tau}\!\left(\sum_{j=0}^{\tau-1} a^{\tau-1-j}\,W_{k+1-\tau+j}\right),
    \qquad \tau \geq 1,
    \label{eqn:est_err_delay_wrong}
\end{equation}
then this implicitly enforces the symmetric policy condition~\eqref{defn:sym_pol}. Under this assumption, the separation principle follows. Indeed, several works in the literature introduce delay into the feedback loop and formulate the estimation error dynamics as in \eqref{eqn:est_err_delay_wrong}~\cite{soleymani2021value, wang2021value, xu2004optimal, soleymani2024foundations, wang2024infinite, maity2018optimal, soleymani2022value}. As shown in problem~$P_3$ and \eqref{eqn:est_error_sep}, this formulation implicitly restricts the policy class to symmetric policies, thereby invoking the separation principle.  

In summary, the majority of existing works on optimal scheduling policies in NCSs restrict attention to symmetric policies, either explicitly or implicitly, thereby enforcing separation. However, the optimality of this class of policies has not been established. In the next section, we prove that symmetric policies are optimal.

\section{Optimality of Symmetric Policies}
To prove the optimality of symmetric policies, we employ DP. We first state the DP recursion. Let the value function at time \(k\) be \(V_k(I_k^{\mathcal S})\), where \(I_k^{\mathcal S}\) denotes the information available at the scheduler, and let \(g_k(I_k^{\mathcal S})\) denote the corresponding stage cost. The DP recursion is
\begin{equation}
    V_k(I_k^{\mathcal S})
    = \min_{D_k \in \{0,1\}}
    \left\{
        g_k(I_k^{\mathcal S})
        + \mathbb{E}\!\left[ V_{k+1}(I_{k+1}^{\mathcal S}) \,\middle|\, I_k^{\mathcal S}, D_k \right]
    \right\}.
    \label{eqn:value_fn}
\end{equation}
Since \(P_3\) imposes a communication-rate constraint, we introduce a remaining budget process \(\{Q_k\}\) defined by
\begin{equation}
    Q_{k+1} = Q_k - D_k,
\end{equation}
with terminal constraint \(Q_N \ge 0\) (equivalently, \(\sum_{k=0}^{N-1} D_k \le N r_s\)). 
Note that \(Q_k\) is determined by \(I_k^{\mathcal S}\). Thus, we do not augment the information state.

Because the controller is delayed by a fixed \(\tau\) steps relative to the scheduler, the effect of a scheduling decision becomes available to the controller only after \(\tau\) time steps. Accordingly, the stage cost can be written as
\begin{align}
    g_k(I_k^{\mathcal S})
    &\triangleq \mathbb{E}\!\left[\mathcal{E}_{k+\tau}^2 \mid I_k^{\mathcal S}\right] \notag\\
    &= \mathbb{E}\!\left[
        \left(
            X_{k+\tau}
            - \mathbb{E}\!\left[ X_{k+\tau} \mid I_{k+\tau}^{\mathcal C} \right]
        \right)^2
        \Big|\, I_k^{\mathcal S}
    \right] \notag\\
    &\overset{(a)}{=}
    \mathbb{E}\!\left[
        \left(
            a^\tau X_k
            + \sum_{j=0}^{\tau-1} a^{\tau-1-j} W_{k+j}
            - \mathbb{E}\!\left[
                a^\tau X_k
                + \sum_{j=0}^{\tau-1} a^{\tau-1-j} W_{k+j}
                \,\Big|\, I_{k+\tau}^{\mathcal C}
            \right]
        \right)^2
        \Big|\, I_k^{\mathcal S}
    \right] \notag\\
    &\overset{(b)}{=}
    \mathbb{E}\!\Bigg[
        a^{2\tau}
        \left(
            X_k
            - \mathbb{E}\!\left[ X_k \mid I_{k+\tau}^{\mathcal C} \right]
        \right)^2
        + \left( \sum_{j=0}^{\tau-1} a^{\tau-1-j} W_{k+j} \right)^2
        \Bigg|\, I_k^{\mathcal S}
    \Bigg] \notag\\
    &\overset{(c)}{=}
    \begin{cases}
        \displaystyle
        \sum_{j=0}^{\tau-1} a^{2(\tau-1-j)} \sigma_W^2,
        & \text{if } k = t_k, \\[1.2ex]
        \displaystyle
        a^{2\tau}
        \left(
            X_k
            - \mathbb{E}\!\left[ X_k \mid I_{k+\tau}^{\mathcal C} \right]
        \right)^2
        + \sum_{j=0}^{\tau-1} a^{2(\tau-1-j)} \sigma_W^2,
        & \text{otherwise.}
    \end{cases}
    \label{eqn:g_k}
\end{align}

where in step (a) we used the state evolution in~\eqref{eqn:X_tk+m} and the fact that \(U_{k+j}\) is \(I^\mathcal{C}_{k+\tau}\)-measurable for all \(0 \le j \le \tau-1\). In step (b), we used the information structure: \(I^\mathcal{C}_{k+\tau}\) depends on \((D_0,\dots,D_k)\), which in turn depends only on past disturbances \((W_0,\dots,W_{k-1})\). Hence, \(W_{k+j}\) is independent of \(I^\mathcal{C}_{k+\tau}\) for all \(j \ge 0\). In step (c), we used the fact that both \(X_k\) and \(\mathbb{E}\!\left[X_k \mid I^\mathcal{C}_{k+\tau}\right]\) are \(I_k^\mathcal{S}\)-measurable. Therefore, the DP recursion becomes
\begin{align}
    V_k(I_k^{\mathcal S})
    = \min_{D_k} \Bigg\{
    & a^{2\tau} \bar{\mathcal{E}}_k^2
      + \mathbb{E}\!\left[ V_{k+1}(I_{k+1}^{\mathcal S}) \mid I_k^{\mathcal S}, D_k=0 \right], \notag\\
    & \mathbb{E}\!\left[ V_{k+1}(I_{k+1}^{\mathcal S}) \mid I_k^{\mathcal S}, D_k=1 \right]
    \Bigg\}
    + \sum_{j=0}^{\tau-1} a^{2(\tau-1-j)} \sigma_W^2 .
    \label{eqn:DP_rec}
\end{align}
where we defined
\begin{equation}
    \bar{\mathcal{E}}_k \triangleq X_k - \mathbb{E}[ X_k \mid I_{k+\tau-1}^{\mathcal C}, U_{k+\tau-1}, Y_{k+\tau}, D_{k}=0 ].
    \label{eqn:E_bar}
\end{equation}
Note that, since $D_k$ is known in~\eqref{eqn:E_bar} (and equals zero), this differs from~\eqref{eqn:est_error_original}, where $D_k$ is a random variable.

The terminal cost for the equivalent problem \(P_2\) is defined as
\begin{equation}
    V_k(I_k^{\mathcal S}) =
    \begin{cases}
        +\infty, & \text{if } Q_k < 0, \\[1.2ex]
        0, & \text{otherwise,}
    \end{cases}
    \qquad \forall\, k \geq N-\tau .
    \label{eqn:terminal_cost}
\end{equation}
This terminal condition is imposed for \(k \ge N-\tau\) because, due to the fixed delay \(\tau\), scheduling decisions \(D_j\) taken after time \(N-\tau\) do not affect the controller within the horizon.

Since there are only two admissible actions at each stage and for each budget level \(Q_k\); we define two action-dependent value functions, denoted by \(V'_{k j d}(I_k^{\mathcal S})\), corresponding to stage \(k\), action \(D_k=d \in \{0,1\}\), and budget \(Q_k=j\).
With this notation, the dynamic programming recursion in~\eqref{eqn:DP_rec} can be written as
\begin{equation}
    V_k(I_k^{\mathcal S})
    = \min_{d \in \{0,1\}} V'_{k j d}(I_k^{\mathcal S}).
    \label{eqn:value_fn_prime}
\end{equation}
Now consider stage \(N-\tau-1\) with remaining budget
\(Q_{N-\tau-1}=1\). Both scheduling actions are feasible under the
inequality constraint. However, this is the final scheduling decision
that can affect the finite-horizon cost, and transmitting an exact state
measurement cannot increase the optimal cost. Therefore, without loss
of optimality, we select \(D_{N-\tau-1}=1\).
Accordingly, the value function is given by
\begin{equation}
    V_{N-\tau-1}(I_{N-\tau-1}^{\mathcal S})\big|_{\,Q_{N-\tau-1}=1}
    = V'_{(N-\tau-1)\,1\,1}(I_{N-\tau-1}^{\mathcal S})
    = \sum_{j=0}^{\tau-1} a^{2(\tau-1-j)} \sigma_W^2 .
    \label{eqn:V_N-tau_Q1}
\end{equation}
At the same stage, but with zero remaining budget \(Q_{N-\tau-1}=0\), the scheduler cannot transmit, and the value function becomes
\begin{equation}
    V_{N-\tau-1}(I_{N-\tau-1}^{\mathcal S})\big|_{\,Q_{N-\tau-1}=0}
    = V'_{(N-\tau-1)\,0\,0}(I_{N-\tau-1}^{\mathcal S})
    = a^{2\tau}\,\bar{\mathcal{E}}_{N-\tau-1}^{\,2}
      + \sum_{j=0}^{\tau-1} a^{2(\tau-1-j)} \sigma_W^2 .
    \label{eqn:V_N-tau_Q0}
\end{equation}

Similarly, we can derive the value functions for stages with sufficient remaining budget.
In particular, these stages correspond to cases where the budget allows scheduling at every subsequent time instant.
This yields the following result.

\begin{lemma}
\label{lemma:DP_Right_Boundry}
Let \(\ell = N-\tau-s\) for some \(s \ge 1\). Then,
\begin{equation}
    V_{\ell}(I_{\ell}^{\mathcal S})\big|_{\,Q_{\ell}=s}
    = V'_{\ell\,s\,1}(I_{\ell}^{\mathcal S})
    = s\!
    \left(
        \sum_{j=0}^{\tau-1} a^{2(\tau-1-j)} \sigma_W^2
    \right),
    \label{eqn:V_l_Q_s+1}
\end{equation}
and
\begin{equation}
     V'_{\ell\,(s-1)\,0}(I_{\ell}^{\mathcal S})
    = a^{2\tau}\bar{\mathcal{E}}_{\ell}^{\,2}
      + s\!
      \left(
          \sum_{j=0}^{\tau-1} a^{2(\tau-1-j)} \sigma_W^2
      \right).
    \label{eqn:V_l-1_Q_s+1}
\end{equation}
\end{lemma}
\begin{proof}
    See Appendix~\ref{Appendix_D}.
\end{proof}

\begin{figure}[t]
    \centering
    \includegraphics[width=\linewidth]{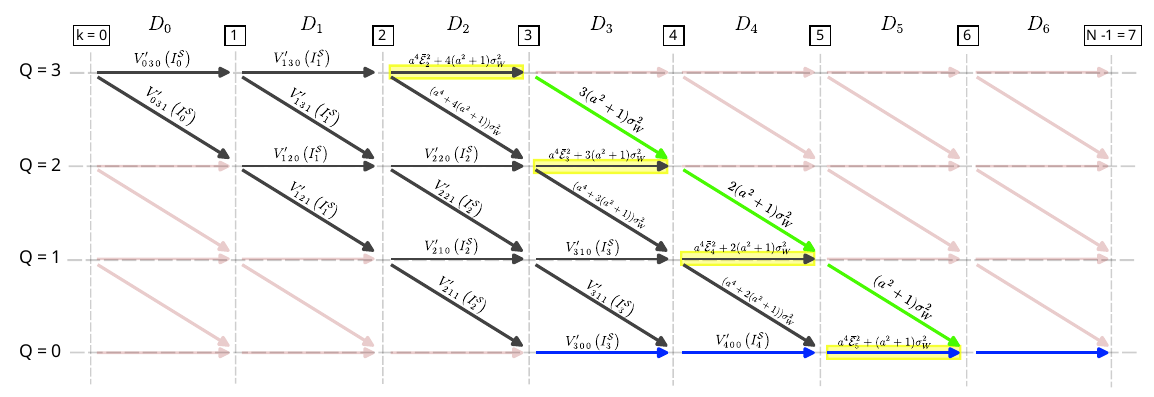}
    \caption{Trellis diagram illustrating the evolution of the value functions across time \(k\) and budget levels \(Q\) for horizon \(N-1=7\), fixed delay \(\tau=2\), and initial budget \(Q_0=3\). At each stage, only two actions are admissible, with the corresponding costs shown on each branch. Infeasible transitions are indicated in red.}
    \label{fig:dp_diagram}
\end{figure}

To better illustrate the cost structures in \eqref{eqn:V_l_Q_s+1} and \eqref{eqn:V_l-1_Q_s+1}, we consider the example trellis diagram shown in Fig.~\ref{fig:dp_diagram}, where the horizon is \(N-1=7\), the delay is fixed to \(\tau=2\), and the initial budget is \(Q_0=3\). Several observations can be made directly from this diagram.

First, note that the estimation error is defined only up to \(k = N-1\) (see the equivalent problem \(P_3\)). Accordingly, the effective horizon is \(N-1\), which already reveals the infeasibility of the delay-free case. Since \(\tau = 2\), the diagram permits transitions to the right by at most one step, thereby demonstrating that zero-delay scheduling is not admissible.

Second, due to the finite budget, once the system reaches the green path, sufficient budget remains to update at every subsequent time step. This path therefore corresponds to the ``always update'' regime, and the associated costs are derived in~\eqref{eqn:V_l_Q_s+1}. In contrast, reaching the blue path indicates that the budget has been fully depleted, which forces the system to follow the ``never update'' regime thereafter. Accordingly, the red branches correspond to infeasible transitions and are excluded from consideration.

Next, the yellow highlighted path corresponds to an intermediate case in which the system chooses not to update at the current stage, while still retaining sufficient budget to update at all subsequent stages. The costs associated with this scenario are derived in~\eqref{eqn:V_l-1_Q_s+1}.

Finally, note that in the trellis diagram shown in Fig.~\ref{fig:dp_diagram}, the cost associated with each branch depends on the time elapsed since the most recent update. Therefore, the underlying state is not fully characterized by the pair $(k, Q)$ alone; rather, an additional state variable is required to capture the time of the most recent update. This effectively introduces an additional dimension to the state space.

We now present the main result of the paper, establishing the optimality of symmetric policies defined in~\eqref{defn:sym_pol}. The proof proceeds by showing that, for the costs associated with the yellow branches, and for any path reaching these branches, the resulting total cost is minimized only under symmetric policies.

\begin{theorem}
\label{thm:sym_policies_are_optimal}
    Consider the system~\eqref{eqn:system_Model}--\eqref{eqn:update_info} with $q,r>0$ and fixed delay $\tau \ge 1$. Under the optimal controller~\eqref{eqn:optimal_controller} and i.i.d.\ zero-mean disturbances, the optimal scheduling policy for the constrained LQR problem $P_2$ is symmetric in the sense of~\eqref{defn:sym_pol}.
\end{theorem}

\begin{proof}
    See Appendix~\ref{Appendix_E}.
\end{proof}

Since we proved the optimality of the symmetric policies, it means separation holds and we can find the closed form solution of the optimal controller.

\begin{lemma}
\label{lemma:est_error_evolution}
    The evolution of the state estimate under a fixed delay $\tau \geq 1$ given as
    \begin{equation}
        \mathbb{E}\!\left[ X_{t_k+m+1} \mid I^{\mathcal C}_{t_k+m+1} \right]
        =
        (a-bL)\,\mathbb{E}\!\left[ X_{t_k+m} \mid I^{\mathcal C}_{t_k+m} \right]
        + \mathbb{E}\!\left[ S_{m+1} \mid I^{\mathcal C}_{t_k+m+1} \right]
        - a\,\mathbb{E}\!\left[ S_m \mid I^{\mathcal C}_{t_k+m} \right]
        \label{eqn:est_error_dynamics}
    \end{equation}
    where $m \geq \tau$.
\end{lemma}

\begin{proof}
    See Appendix~\ref{Appendix_F}.
\end{proof}

\begin{remark}
    Under symmetric policies defined in~\eqref{defn:sym_pol}, the conditional expectations of the accumulated disturbances satisfy
    \(\mathbb{E}[S_{m+1}\mid I^{\mathcal C}_{t_k+m+1}] = a\,\mathbb{E}[S_m\mid I^{\mathcal C}_{t_k+m}]\).
    Consequently, the evolution of the optimal controller reduces to
    \begin{equation}
        \mathbb{E}\!\left[ X_{t_k+m+1} \mid I^{\mathcal C}_{t_k+m+1} \right]
        =
        (a-bL)\,\mathbb{E}\!\left[ X_{t_k+m} \mid I^{\mathcal C}_{t_k+m} \right].
        \label{eqn:sym_controller}
    \end{equation}
\end{remark}

\noindent Therefore, using \eqref{eqn:optimal_controller}, \eqref{eqn:X_tk+m}, and \eqref{eqn:sym_controller}, we obtain that for all \(m \ge \tau\),
\begin{equation}
    U_{t_k+m}
    = -L\,\mathbb{E}\!\left[ X_{t_k+m} \mid I^{\mathcal C}_{t_k+m} \right]
    = -L\,(a-bL)^{m-\tau}
    \left(
    a^{\tau} X_{t_k}
    +
    \sum_{j=0}^{\tau-1} a^{\tau-1-j} b\,U_{t_k+j}
    \right).
    \label{eqn:opt_controller_sym}
\end{equation}
This expression shows that, under symmetric policies, the controller evolves as a purely linear time-invariant system with closed-loop feedback gain \((a-bL)\) after the delay period. The control input at time \(t_k+m\) depends only on the most recent state update \(X_{t_k}\) and the finite sequence of control inputs applied during the delay interval \(\{t_k,\dots,t_k+\tau-1\}\). Importantly, no additional information is conveyed through the absence of transmissions, and the controller dynamics are independent of the scheduling decisions beyond their effect on the update times.

In the next section, we establish that the optimal scheduling policy is a symmetric threshold policy in the accumulated noise~\eqref{defn:S_m}. 
Before proceeding, however, we first show that, under a noise distribution symmetric about zero and symmetric scheduling policies~\eqref{defn:sym_pol}, the controller’s conditional expectation of the noise sequence $\{W_k\}$ remains zero.

\begin{theorem}
\label{thm:symmetric_W}
Consider the system~\eqref{eqn:system_Model}--\eqref{eqn:update_info} with $q,r>0$ and fixed delay $\tau \ge 1$.
Assume $\{W_k\}$ is an i.i.d. noise sequence with a distribution symmetric about zero.
Under symmetric scheduling policies~\eqref{defn:sym_pol}, the controller's conditional estimate of noise remains zero, i.e.,
\begin{equation}
\mathbb{E}[W_{t_k+j} \mid I^{\mathcal C}_{t_k+m}] = 0
\quad \forall\, j,m \ge 0 .
\label{eqn:sym_W}
\end{equation}
\end{theorem}

\begin{proof}
    Under symmetric scheduling policies~\eqref{defn:sym_pol}, the problem reduces to $P_3$ with estimation error dynamics (under the absence of an update)
    \begin{equation}
        \mathcal{E}_{k+1} = a \mathcal{E}_k + W_k,
        \label{eqn:sym_est_error_dyn}
    \end{equation}
    where $\{W_k\}$ is an i.i.d. noise sequence symmetric about zero.
    Let the value function be defined as in~\eqref{eqn:value_fn}.  
    
    We first show that $V_k(\epsilon)$ is an even function of $\epsilon$, i.e.,
    \[
    V_k(\epsilon) = V_k(-\epsilon).
    \]
    
    \noindent \textit{Step 1: Terminal condition.}  
    From~\eqref{eqn:terminal_cost}, the terminal cost is quadratic in $\mathcal{E}_N$, hence even.
    
    \noindent \textit{Step 2: Induction hypothesis.}  
    Assume $V_{k+1}(\cdot)$ is even.
    
    \noindent Define the action-dependent cost-to-go
    \begin{equation}
        \hat V_k(\epsilon,d)
        =
        \epsilon^2
        +
        \mathbb{E}\!\left[
            V_{k+1}(\mathcal{E}_{k+1})
            \mid \mathcal{E}_k=\epsilon, d
        \right].
    \end{equation}
    
    \noindent Because $W_k$ is symmetric about zero, we have \(W_k \overset{d}{=} -W_k. \)
    Hence,
    \begin{equation}
    a(-\epsilon) + W_k
    = -a\epsilon + W_k
    \overset{d}{=}
    -a\epsilon - W_k
    = -(a\epsilon + W_k).
        \label{eqn:sym_of_est_error}
    \end{equation}\
    
    \noindent We now compute:
    \begin{align}
        \hat{V}_k(-\epsilon, d) & = \epsilon^2 + \mathbb{E}[V_{k+1}(\mathcal{E}_{k+1}) \mid \mathcal{E}_k=-\epsilon, d] \notag \\
        & = \epsilon^2 + \mathbb{E}[V_{k+1}(a\;\mathcal{E}_k+W_k) \mid \mathcal{E}_k=-\epsilon, d] \notag \\
        & = \epsilon^2 + \mathbb{E}[V_{k+1}(a\;(-\mathcal{E}_k)+W_k) \mid \mathcal{E}_k=\epsilon, d] \notag \\
        & \overset{(a)}{=} \epsilon^2 + \mathbb{E}[V_{k+1}(-(a\;\mathcal{E}_k+W_k)) \mid \mathcal{E}_k=\epsilon, d] \notag \\
        & = \epsilon^2 + \mathbb{E}[V_{k+1}(-\mathcal{E}_{k+1}) \mid \mathcal{E}_k=\epsilon, d] \notag \\
        & \overset{(b)}{=}  \epsilon^2 + \mathbb{E}[V_{k+1}(\mathcal{E}_{k+1}) \mid \mathcal{E}_k=\epsilon, d] \notag \\
        & = \hat{V}_k(\epsilon, d)
        \label{eqn:sym_of_value_fn}
    \end{align}
    where (a) follows from the distributional symmetry of the estimation error dynamics in~\eqref{eqn:sym_of_est_error}, and (b) follows from the induction hypothesis that $V_{k+1}(\cdot)$ is an even function. Thus $\hat V_k(\epsilon,d)$ is even in $\epsilon$ for every $d$. 
    Since
    \[
    V_k(\epsilon)
    =
    \min_d \hat V_k(\epsilon,d),
    \]
    it follows that $V_k(\epsilon)$ is also even. Hence, the optimal action is
    \begin{equation}
    	D^*_k(\epsilon) \triangleq \arg \min_d V_k(\epsilon,d)
    \end{equation}
    so because of \eqref{eqn:sym_of_value_fn}
    \(
    D^*_k(\epsilon) = D^*_k(-\epsilon)
    \)
    which means that there exists an optimal policy that is even in the
    estimation error. Under~\eqref{eqn:sym_est_error_dyn}, simultaneous
    sign reversal of the estimation error and disturbance produces the
    sign-reversed error trajectory while leaving all scheduling decisions
    unchanged. Since the disturbances are i.i.d.\ and symmetric about zero,
    their conditional distributions given the resulting scheduling history
    remain symmetric. Therefore,
    \(\mathbb{E}[W_{t_k+j}\mid I_{t_k+m}^{\mathcal C}]=0\) for
    \(0\le j<m\), which establishes~\eqref{eqn:sym_W}.
\end{proof}

\section{Optimal Scheduling Policy}
Having established the optimality of symmetric scheduling policies and derived the corresponding optimal controller, we now characterize the optimal scheduling policy by solving the associated DP. We first consider the value function under zero budget, which corresponds to the blue path in Fig.~\ref{fig:dp_diagram}.
\begin{lemma}
\label{lem:zero_budget_recursion}
    Suppose that, for \(Q_k=0\) and \(D_k=0\),
    \begin{equation}
        V'_{k\,0\,0}\!\left(I_k^{\mathcal S}\right)
        =
        s_{k0}\,\bar{\mathcal{E}}_k^{2}
        +
        c_{k00}\,\sigma_W^{2}, \qquad \forall \; k\leq N-\tau-1.
    \end{equation}
    Then, for \(Q_{k-1}=0\) and \(D_{k-1}=0\),
    \begin{equation}
        V'_{(k-1)\,0\,0}\!\left(I_{k-1}^{\mathcal S}\right)
        =
        s_{(k-1)0}\,\bar{\mathcal{E}}_{k-1}^{2}
        +
        c_{(k-1)00}\,\sigma_W^{2},
        \label{eqn:V_k-1_00}
    \end{equation}
    where
    \begin{align}
        s_{(k-1)0} &\triangleq a^{2\tau} + a^{2}s_{k0}, \\
        c_{(k-1)00} &\triangleq s_{k0} + c_{k00} + \sum_{j=0}^{\tau-1} a^{2(\tau-1-j)} .
    \end{align}
\end{lemma}

\begin{proof}
    We recall that
    \begin{equation}
        V'_{(N-\tau-1)\,0\,0}\!\left(I^{\mathcal S}_{N-\tau-1}\right)
        =
        a^{2\tau}\,\bar{\mathcal E}_{N-\tau-1}^{2}
        +
        \sum_{j=0}^{\tau-1} a^{2(\tau-1-j)} \sigma_W^{2}.
        \label{eqn:V_N-tau_00}
    \end{equation}
    We now proceed with the induction step. Using the DP recursion in \eqref{eqn:DP_rec}, we obtain
    \begin{align}
        V'_{(k-1)\,0\,0}\!\left(I_{k-1}^{\mathcal S}\right)
        &= a^{2\tau} \bar{\mathcal{E}}_{k-1}^2
        + \mathbb{E}\!\left[
        s_{k0}\,\bar{\mathcal{E}}_k^{2}
        + c_{k00}\,\sigma_W^{2}
        \;\middle|\;
        I^{\mathcal S}_{k-1},\; D_{k-1}=0
        \right]
        + \sum_{j=0}^{\tau-1} a^{2(\tau-1-j)} \sigma_W^2 \notag\\[1ex]
        & \overset{(a)}{=}
        a^{2\tau}\,\bar{\mathcal{E}}_{k-1}^{2}
        + s_{k0}\big(a^{2}\bar{\mathcal{E}}_{k-1}^{2}+\sigma_W^{2}\big)
        + c_{k00}\,\sigma_W^{2}
        + \sum_{j=0}^{\tau-1} a^{2(\tau-1-j)} \sigma_W^2 \notag\\
        &=
        \big(a^{2\tau}+a^{2}s_{k0}\big)\bar{\mathcal{E}}_{k-1}^{2}
        + \bigg(
        s_{k0}+c_{k00}
        + \sum_{j=0}^{\tau-1} a^{2(\tau-1-j)}
        \bigg)\sigma_W^{2} \notag\\
        &=
        s_{(k-1)0}\,\bar{\mathcal{E}}_{k-1}^{2}
        + c_{(k-1)00}\,\sigma_W^{2}.
    \end{align}
    In step~(a), we used the system dynamics \eqref{eqn:system_Model} and the fact that under zero budget no scheduling occurs, so the estimation error evolves open-loop with additive disturbance variance~\(\sigma_W^2\).
\end{proof}

Next, we characterize the value function when the remaining budget is one and the scheduler is activated, i.e., when \(Q_{k-1}=1\) and \(D_{k-1}=1\), which leaves zero budget for all subsequent stages. In the example shown in Fig.~\ref{fig:dp_diagram}, these costs correspond to
\(V'_{2\,1\,1}\!\left(I_2^{\mathcal S}\right)\) and
\(V'_{3\,1\,1}\!\left(I_3^{\mathcal S}\right)\).
\begin{lemma}
\label{lemma:D_1_Q_1}
    Given that
    \begin{equation}
        V'_{k\,0\,0}\left(I_k^{\mathcal S}\right)
        =
        s_{k0}\,\bar{\mathcal{E}}_k^{2}
        +
        c_{k00}\,\sigma_W^{2}, \qquad \forall \; k \leq N - \tau -1
    \end{equation}
    the value function at the previous stage under \(Q_{k-1}=1\) and \(D_{k-1}=1\) is given by
    \begin{equation}
       V'_{(k-1)\,1\,1}\left(I_{k-1}^{\mathcal S}\right)
        =
        c_{(k-1)11}\,\sigma_W^{2},
    \end{equation}
    where 
    \begin{equation}
        c_{(k-1)11} \triangleq s_{k0} + c_{k00} + \left(\sum_{j=0}^{\tau-1} a^{2(\tau-1-j)} \right)
        = c_{(k-1)00}.
    \end{equation}
\end{lemma}

\begin{proof}
    We know that
    \begin{equation}
        V'_{(N-\tau-1)\,0\,0}\!\left(I^{\mathcal S}_{N-\tau-1}\right)
        =
        a^{2\tau}\,\bar{\mathcal E}_{N-\tau-1}^{2}
        +
        \sum_{j=0}^{\tau-1} a^{2(\tau-1-j)} \sigma_W^{2}.
    \end{equation}
    Consequently, using the DP recursion~\eqref{eqn:DP_rec}, we obtain
    \begin{equation}
        V'_{(N-\tau-2)\,1\,1}\!\left(I^{\mathcal S}_{N-\tau-2}\right)
        =
        \left(
            a^{2\tau}
            +
            2\sum_{j=0}^{\tau-1} a^{2(\tau-1-j)}
        \right)\sigma_W^{2},
    \end{equation}
    where we used the fact that the decision \(D_{N-\tau-1}\) is state independent. 
    We now proceed with the induction step. For a general stage \(k\), using \eqref{eqn:V_k-1_00}, we have
    \begin{align}
        V'_{(k-1)\,1\,1}\!\left(I^{\mathcal S}_{k-1}\right)
        &=
        \mathbb{E}\!\left[
            s_{k0}\,\bar{\mathcal E}_k^{2}
            + c_{k00}\,\sigma_W^{2}
            \;\middle|\;
            I^{\mathcal S}_{k-1},\, D_{k-1}=1
        \right]
        + \sum_{j=0}^{\tau-1} a^{2(\tau-1-j)} \sigma_W^{2} \notag\\
        &=
        s_{k0}\,
        \mathbb{E}\!\left[
            \bar{\mathcal E}_k^{2}
            \;\middle|\;
            I^{\mathcal S}_{k-1},\, D_{k-1}=1
        \right]
        +
        \left(
            c_{k00}
            +
            \sum_{j=0}^{\tau-1} a^{2(\tau-1-j)}
        \right)\sigma_W^{2}.
    \end{align}
    Now expand the error term under an update:
    \begin{align}
    \mathbb{E}\!\left[
    \bar{\mathcal E}_k^{2}
    \;\middle|\;
    I^{\mathcal S}_{k-1},\, D_{k-1}=1
    \right]
    &=
    \mathbb{E}\!\left[
    \Big(a X_{k-1} + W_{k-1}
    - \mathbb{E}[a X_{k-1}+W_{k-1}\mid I^{\mathcal C}_{k+\tau}]
    \Big)^{2}
    \;\middle|\;
    I^{\mathcal S}_{k-1},\, D_{k-1}=1
    \right] \notag\\
    &=
    \mathbb{E}\!\left[
    \Big(W_{k-1}\Big)^{2}
    \;\middle|\;
    I^{\mathcal S}_{k-1},\, D_{k-1}=1
    \right] \notag\\
    &= \mathbb{E}[W_{k-1}^{2}] = \sigma_W^{2},
    \end{align}
    where we used the causality of the scheduling policy, the independence and zero-mean property of \(W_{k-1}\), and the fact that \(D_{k-1}=1\) implies \(X_{k-1}\in I^{\mathcal C}_{k+\tau}\), so that
    \(\mathbb{E}[aX_{k-1}+W_{k-1}\mid I^{\mathcal C}_{k+\tau}] = aX_{k-1}\).
    Hence,
    \begin{equation}
    V'_{(k-1)\,1\,1}\!\left(I^{\mathcal S}_{k-1}\right)
    =
    \left(
    s_{k0}+c_{k00}+ \sum_{j=0}^{\tau-1} a^{2(\tau-1-j)}
    \right)\sigma_W^{2}
    \;\triangleq\;
    c_{(k-1)11}\,\sigma_W^{2}.
    \end{equation}
\end{proof}

To complete the DP solution, two additional induction arguments are required. 
The first induction corresponds to the \emph{diagonal} path (e.g., the sequence
$V'_{1\,3\,1}\!\left(I_1^{\mathcal S}\right)$,
$V'_{2\,2\,1}\!\left(I_2^{\mathcal S}\right)$, and
$V'_{3\,1\,1}\!\left(I_3^{\mathcal S}\right)$ in Fig.~\ref{fig:dp_diagram}). 
The second induction corresponds to the \emph{horizontal} path in the Trellis Diagram (e.g., the sequence
$V'_{2\,1\,0}\!\left(I_2^{\mathcal S}\right)$ and
$V'_{3\,1\,0}\!\left(I_3^{\mathcal S}\right)$ in Fig.~\ref{fig:dp_diagram}).

\begin{lemma}
\label{lemma:Diagonal}
    Suppose that
    \begin{equation}
        V'_{(k+1)j0}\!\left(I^{\mathcal S}_{k+1}\right)
        =
        s_{(k+1)j}\,\bar{\mathcal E}_{k+1}^{2}
        +
        c_{(k+1)j0}\,\sigma_W^{2} + z_{(k+1)j0},
    \end{equation}
    and
    \begin{equation}
        V'_{(k+1)j1}\!\left(I^{\mathcal S}_{k+1}\right)
        =
        c_{(k+1)j1}\,\sigma_W^{2} + z_{(k+1)j1}.
    \end{equation}
    and assume symmetric policies. Then,
    \begin{equation}
        V'_{k(j-1)1}\!\left(I^{\mathcal S}_{k}\right)
        =
        c_{k(j-1)1}\,\sigma_W^{2} + z_{k(j-1)1},
        \label{eqn:V_k_j-1_1}
    \end{equation}
    where
    \begin{equation}
    \begin{aligned}
        c_{k(j-1)1}
        & \triangleq
        c_{(k+1)j0}
        \Pr\!\big(D_{k+1}=0 \mid D_k=1\big) + c_{(k+1)j1}
        \Pr\!\big(D_{k+1}=1 \mid D_k=1\big)
        +
        \sum_{i=0}^{\tau-1} a^{2(\tau-1-i)} \notag \\
        z_{k(j-1)1} & \triangleq \Pr\!\big(D_{k+1}=0 \mid D_k=1\big) \left(z_{(k+1)j0} + s_{(k+1)j} \; \mathbb{E}\left[ W_k^2 \mid I_k^{\mathcal S},\, D_k=1,\, D_{k+1}=0\right] \right) \notag \\
        & \qquad + \Pr\!\big(D_{k+1}=1 \mid D_k=1\big) z_{(k+1)j1}.
    \end{aligned}
    \end{equation}
\end{lemma}

\begin{proof}
The initial steps have been established in the preceding discussion, hence, we only prove the induction step here.
\begin{align}
    V'_{k(j-1)1}\!\left(I_k^{\mathcal S}\right)
    &=
    \mathbb{E}\!\left[
        V_{k+1}\!\left(I^{\mathcal S}_{k+1}\right)
        \,\middle|\,
        I_k^{\mathcal S},\, D_k=1
    \right]
    +
    \left(\sum_{i=0}^{\tau-1} a^{2(\tau-1-i)}\right)\sigma_W^2
    \notag\\[0.5em]
    &=
    \Pr\!\big(D_{k+1}=0 \mid D_k=1\big)
    \mathbb{E}\!\left[
        s_{(k+1)j}\,\bar{\mathcal E}_{k+1}^2
        +
        c_{(k+1)j0}\,\sigma_W^2 + z_{(k+1)j0}
        \,\middle|\,
        I_k^{\mathcal S},\, D_k=1,\, D_{k+1}=0
    \right]
    \notag\\
    &\quad
    +
    \Pr\!\big(D_{k+1}=1 \mid D_k=1\big)
    \mathbb{E}\!\left[
        c_{(k+1)j1}\,\sigma_W^2 + z_{(k+1)j1}
        \,\middle|\,
        I_k^{\mathcal S},\, D_k=1,\, D_{k+1}=1
    \right] \notag \\
    & \qquad +
    \left(\sum_{i=0}^{\tau-1} a^{2(\tau-1-i)}\right)\sigma_W^2
    \notag\\[0.5em]
    &=
    s_{(k+1)j}\Pr\!\big(D_{k+1}=0 \mid D_k=1\big)
    \mathbb{E}\!\left[
        \bar{\mathcal E}_{k+1}^2
        \,\middle|\,
        I_k^{\mathcal S},\, D_k=1,\, D_{k+1}=0
    \right]
    \notag\\
    &\quad
    + \left(\Pr\!\big(D_{k+1}=0 \mid D_k=1\big)z_{(k+1)j0} + \Pr\!\big(D_{k+1}=1 \mid D_k=1\big)z_{(k+1)j1}\right) \notag \\
    &\quad
    +
    \Bigg(
        \Pr\!\big(D_{k+1}=0 \mid D_k=1\big)c_{(k+1)j0}
        +
        \Pr\!\big(D_{k+1}=1 \mid D_k=1\big)c_{(k+1)j1} \notag \\
        & \qquad +
        \left(\sum_{i=0}^{\tau-1} a^{2(\tau-1-i)}\right)
    \Bigg)\sigma_W^2.
\end{align}
Moreover,
\begin{align}
    &\mathbb{E}\!\left[
        \bar{\mathcal E}_{k+1}^{2}
        \,\middle|\,
        I_k^{\mathcal S},\, D_k=1,\, D_{k+1}=0
    \right] \notag \\
    & \qquad =
    \mathbb{E}\!\Bigg[
        \left(
            aX_k + bU_k + W_k
            -
            \mathbb{E}\!\left[
                aX_k + bU_k + W_k
                \mid
                I_{k+\tau-1}^{\mathcal C},\, D_k=1,\, D_{k+1}=0
            \right]
        \right)^2 \notag \\
        & \qquad \qquad \qquad \qquad \qquad \Bigg|\,
        I_k^{\mathcal S},\, D_k=1,\, D_{k+1}=0
    \Bigg]
    \notag\\
    & \qquad =
    \mathbb{E}[W_k^2 \,\mid\,
        I_k^{\mathcal S},\, D_k=1,\, D_{k+1}=0],
\end{align}
where the last equality follows since \(X_k\) is known when \(D_k=1\), $U_k \in I_{k+\tau}^{\mathcal C}$, the policy is symmetric, and Theorem \ref{thm:symmetric_W}.
Substituting this result yields eqn.~\eqref{eqn:V_k_j-1_1} which completes the induction step.
\end{proof}

\begin{lemma}
\label{lemma:Horizontal}
    Given that
    \begin{equation}
        V'_{kj0}\!\left(I_k^{\mathcal S}\right)
        =
        s_{kj}\,\bar{\mathcal E}_{k}^{2}
        +
        c_{k j 0}\,\sigma_W^{2} + z_{kj0},
        \label{eqn:V_kj0}
    \end{equation}
    and
    \begin{equation} 
        V'_{kj1}\!\left(I_k^{\mathcal S}\right)
        =
        c_{k j 1}\,\sigma_W^{2} + z_{kj1}.
        \label{eqn:V_kj1}
    \end{equation}
    Under symmetric policies,
    \begin{equation}
        V'_{(k-1)j0}\!\left(I^{\mathcal S}_{k-1}\right)
        =
        s_{(k-1)j}\,\bar{\mathcal E}_{k-1}^{2}
        +
        c_{(k-1)j0}\,\sigma_W^{2} + z_{(k-1)j0},
        \label{eqn:V_k-1_j0}
    \end{equation}
    where
    \begin{align}
        s_{(k-1)j}
        &\triangleq
        a^{2\tau} + a^{2}\Pr\!\big(D_k=0 \mid D_{k-1}=0\big) s_{kj},
        \\[0.3em]
        c_{(k-1)j0}
        &\triangleq
        \Pr\!\big(D_k=0 \mid D_{k-1}=0\big)c_{k j 0}
        +
        \Pr\!\big(D_k=1 \mid D_{k-1}=0\big)c_{k j 1} +
        \left(\sum_{i=0}^{\tau-1} a^{2(\tau-1-i)}\right) 
        \notag \\
        z_{(k-1)j0} & \triangleq \Pr\!\big(D_k=0 \mid D_{k-1}=0\big) \left(z_{kj0} + s_{kj} \; \mathbb{E}\left[W_{k-1}^2 \mid  I^{\mathcal S}_{k-1},\, D_{k-1}=0,\, D_k=0 \right] \right) \notag \\
        & \quad + \Pr\!\big(D_k=1\mid D_{k-1}=0\big) z_{kj1} 
        .
    \end{align}
\end{lemma}

\begin{proof}
    Since \(Q_{k-1}=j\) is symmetric, we have
    \begin{align}
        V'_{(k-1)j0}\!\left(I^{\mathcal S}_{k-1}\right)
        &=
        a^{2\tau}\,\bar{\mathcal E}_{k-1}^{2}
        +
        \Pr\!\big(D_k=0 \mid D_{k-1}=0\big)
        \mathbb{E}\!\left[
            s_{kj}\,\bar{\mathcal E}_{k}^{2}
            +
            c_{k j 0}\,\sigma_W^{2} + z_{kj0}
            \,\middle|\,
            I^{\mathcal S}_{k-1},\, D_{k-1}=0,\, D_k=0
        \right]
        \notag\\
        &\quad
        +
        \Pr\!\big(D_k=1 \mid D_{k-1}=0\big)
        \left( c_{k j 1}\,\sigma_W^{2} + z_{kj1} \right)
        +
        \left(\sum_{i=0}^{\tau-1} a^{2(\tau-1-i)}\right)\sigma_W^{2}.
    \end{align}
    Pulling deterministic terms outside the expectations yields
    \begin{align}
        V'_{(k-1)j0}\!\left(I^{\mathcal S}_{k-1}\right)
        & =
        a^{2\tau}\,\bar{\mathcal E}_{k-1}^{2}
        +
        \Pr\!\big(D_k=0 \mid D_{k-1}=0\big)
        \,s_{kj}
        \,\mathbb{E}\!\left[
            \bar{\mathcal E}_{k}^{2}
            \,\middle|\,
            I^{\mathcal S}_{k-1},\, D_{k-1}=0,\, D_k=0
        \right] \notag \\
        & \quad + \left( \Pr\!\big(D_k=0 \mid D_{k-1}=0\big)z_{k j 0} + \Pr\!\big(D_k=1 \mid D_{k-1}=0\big)z_{k j 1}\right) \notag \\
        & \quad +
        \Bigg(
            \Pr\!\big(D_k=0 \mid D_{k-1}=0\big)c_{k j 0}
            +
            \Pr\!\big(D_k=1 \mid D_{k-1}=0\big)c_{k j 1} \notag \\
            & \quad +
            \left(\sum_{i=0}^{\tau-1} a^{2(\tau-1-i)}\right)
        \Bigg)\sigma_W^{2}.
    \end{align}
    where
    \begin{align}
        & \mathbb{E}\!\left[
            \bar{\mathcal E}_{k}^{2}
            \,\middle|\,
            I^{\mathcal S}_{k-1},\, D_{k-1}=0,\, D_k=0
        \right] \notag \\
        &=
        \mathbb{E}\!\left[
            \left(
                aX_{k-1} + W_{k-1}
                -
                \mathbb{E}\!\left[
                    aX_{k-1} + W_{k-1}
                    \mid
                    I^{\mathcal C}_{k+\tau-2},\, D_{k-1}=0,\, D_k=0
                \right]
            \right)^2
            \,\middle|\,
            I^{\mathcal S}_{k-1},\, D_{k-1}=0,\, D_k=0
        \right]
        \notag\\
        & \overset{(a)}{=}
        a^{2}\,
        \mathbb{E}\!\left[
            \left(
                X_{k-1}
                -
                \mathbb{E}\!\left[
                    X_{k-1}
                    \mid
                    I^{\mathcal C}_{k+\tau-2},\, D_{k-1}=0
                \right]
            \right)^2
            \,\middle|\,
            I^{\mathcal S}_{k-1},\, D_{k-1}=0
        \right] \notag \\
        & \qquad +
        \mathbb{E}\left[W_{k-1}^2 \mid I^{\mathcal S}_{k-1},\, D_{k-1}=0,\, D_k=0\right]
        \notag\\
        &=
        a^{2}\,\bar{\mathcal E}_{k-1}^{2}
        +
        \mathbb{E}\left[W_{k-1}^2 \mid I^{\mathcal S}_{k-1},\, D_{k-1}=0,\, D_k=0\right].
    \end{align}
    where the step (a) is because Theorem \ref{thm:symmetric_W} and Prop.\ref{prop:Subset}. Substituting this result establishes \eqref{eqn:V_k-1_j0}.
\end{proof}

Note that the optimal scheduling policy is given by
\begin{equation}
    D_k
    =
    \mathds{1}_{\Big\{
        V'_{kj0}\!\left(I_k^{\mathcal S}\right)
        \; \ge \;
        V'_{kj1}\!\left(I_k^{\mathcal S}\right)
    \Big\}}.
\end{equation}
Hence, using \eqref{eqn:V_kj0} and \eqref{eqn:V_kj1}, the optimal policy can be written as
\begin{equation}
    D_k
    =
    \mathds{1}_{\Big\{
        \bar{\mathcal E}_{k}^{2}
        \; \ge \;
        \frac{(c_{k j 1}-c_{k j 0}) + (z_{kj1} - z_{kj0})/\sigma_W^2}{s_{kj}}\,\sigma_W^{2}
    \Big\}}.
\end{equation}
This is a symmetric threshold policy on the estimation error.

Now, for \(k=t_k+m\), we have
\begin{align}
    \bar{\mathcal E}_k
    &=
    X_k
    -
    \mathbb{E}\!\left[
        X_k
        \mid
        I_{k+\tau-1}^{\mathcal C},\,
        U_{k+\tau-1},\,
        Y_{k+\tau},\,
        D_k=0
    \right]
    \notag\\
    &\overset{(a)}{=}
    \sum_{j=0}^{m-1} a^{m-1-j} W_{t_k+j}
    -
    \sum_{j=0}^{m-1} a^{m-1-j}
    \mathbb{E}\!\left[
        W_{t_k+j}
        \mid
        I_{t_k+\tau+m}^{\mathcal C}
    \right]
    \notag\\
    &\overset{(b)}{=}
    S_m,
\end{align}
where step (a) follows from the system dynamics in \eqref{eqn:system_Model} and the fact that the control inputs are known given \(I_{k+\tau-1}^{\mathcal C}\), and step (b) follows from the symmetry of the scheduling policy.
Consequently, the optimal scheduling policy can be expressed as
\begin{equation}
    D_k
    =
    \mathds{1}_{\Big\{
        S_m^{2}
        \; \ge \;
        \alpha_{kj}\,\sigma_W^{2}
    \Big\}},
    \label{eqn:opt_scheduling_pol}
\end{equation}
where
\(
    \alpha_{kj}
    \triangleq
    \frac{(c_{k j 1}-c_{k j 0}) + (z_{kj1} - z_{kj0})/\sigma_W^2}{s_{kj}},
\)
and \(S_m\) denotes the accumulated disturbance as defined in \eqref{defn:S_m}.

\begin{figure}[t]
    \centering
    \includegraphics[width=0.8\linewidth]{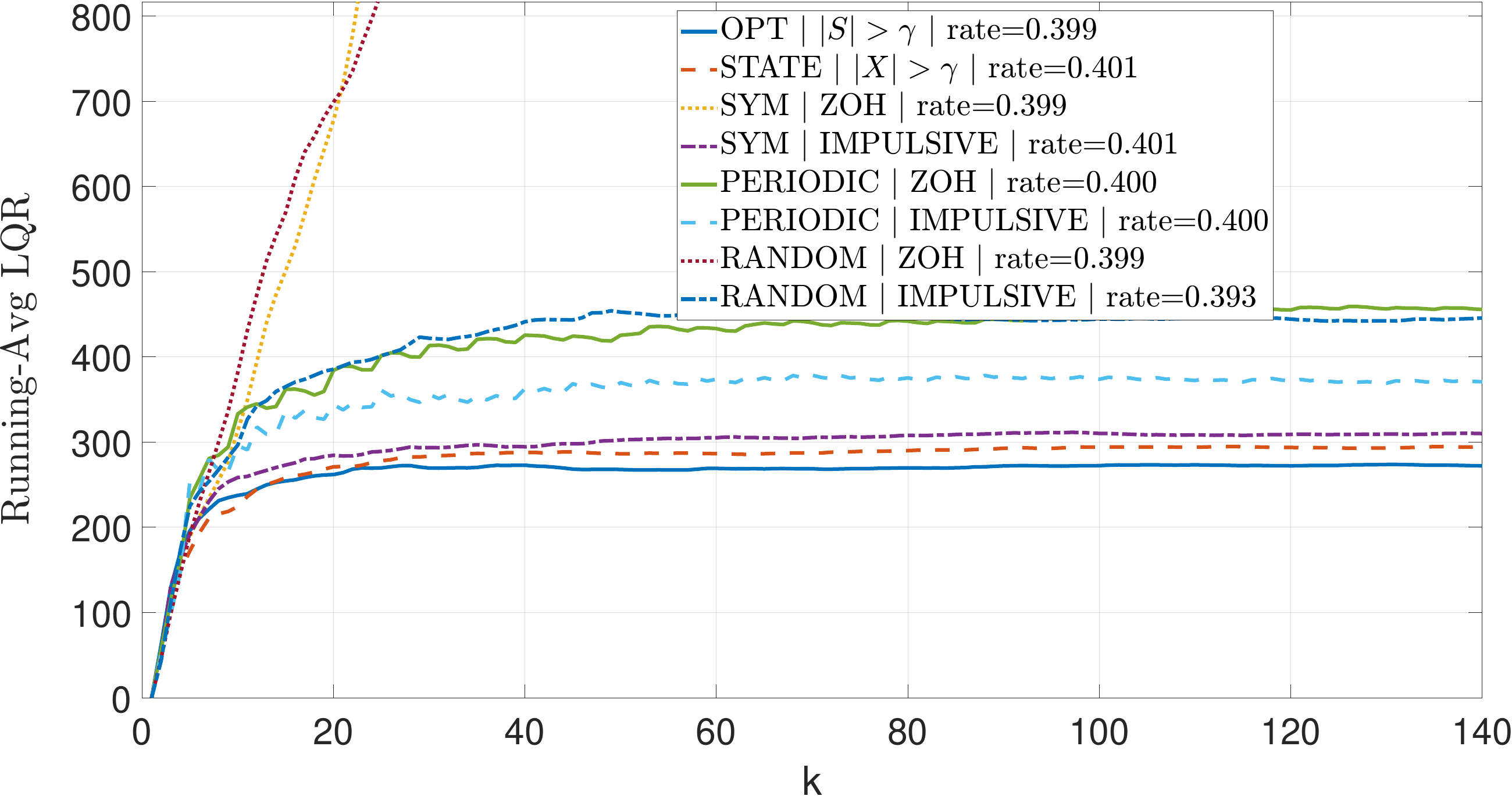}
    \caption{Running-average LQR cost versus time under Gaussian disturbances. All policies are tuned to satisfy the same target communication rate $r_s = 0.4$. The optimal symmetric threshold policy (OPT) achieves the lowest cost, followed closely by symmetric policies, while periodic and random strategies incur higher cost. Results are averaged over 100 Monte Carlo runs.}
    \label{fig:Single_LQR}
\end{figure}

\section{Numerical Results}
In this section, we evaluate the performance of the optimal policy 
\eqref{eqn:opt_controller_sym},\eqref{eqn:opt_scheduling_pol} 
and compare it against commonly used sub-optimal controller and scheduling policy combinations. 
As established earlier, the optimal policy applies to any zero-mean disturbance distribution symmetric about the origin and does not rely on Gaussian assumptions.  We therefore evaluate its performance under multiple symmetric noise models.

As benchmark controller policies, we implement two widely used alternatives: 
\emph{zero-order hold (ZOH)} and \emph{impulsive (IMP)} control.

The ZOH controller applies the linear feedback law~\eqref{eqn:optimal_controller} at the observation time and then holds the control input constant until the next observation. Specifically,
\begin{equation}
    U^{\mathrm{ZOH}}_k =
    \begin{cases}
        -L \hat{X}_k, & if \quad k = t_k + \tau, \\
        U_{k-1}, & \text{otherwise}.
    \end{cases}
    \label{eqn:ZOH_controller}
\end{equation}
In contrast, the impulsive controller concentrates the entire control action at the observation instant and applies no control between update times:
\begin{equation}
    U^{\mathrm{IMP}}_k =
    \begin{cases}
        -\dfrac{a}{b}\,\hat{X}_k, & if \quad  k = t_k + \tau, \\
        0, & \text{otherwise}.
    \end{cases}
    \label{eqn:imp_controller}
\end{equation}
Since the impulsive controller exerts no input during the inter-scheduling interval, the feedback gain is chosen as $-a/b$ rather than $-L$ to ensure a fair comparison. Note that this gain corresponds to the limiting case of the LQR controller as $r \to 0$, where control effort is penalized negligibly and the state is driven to zero in a single step.

\begin{figure}[!t]
    \centering
    \includegraphics[width=0.8\linewidth]{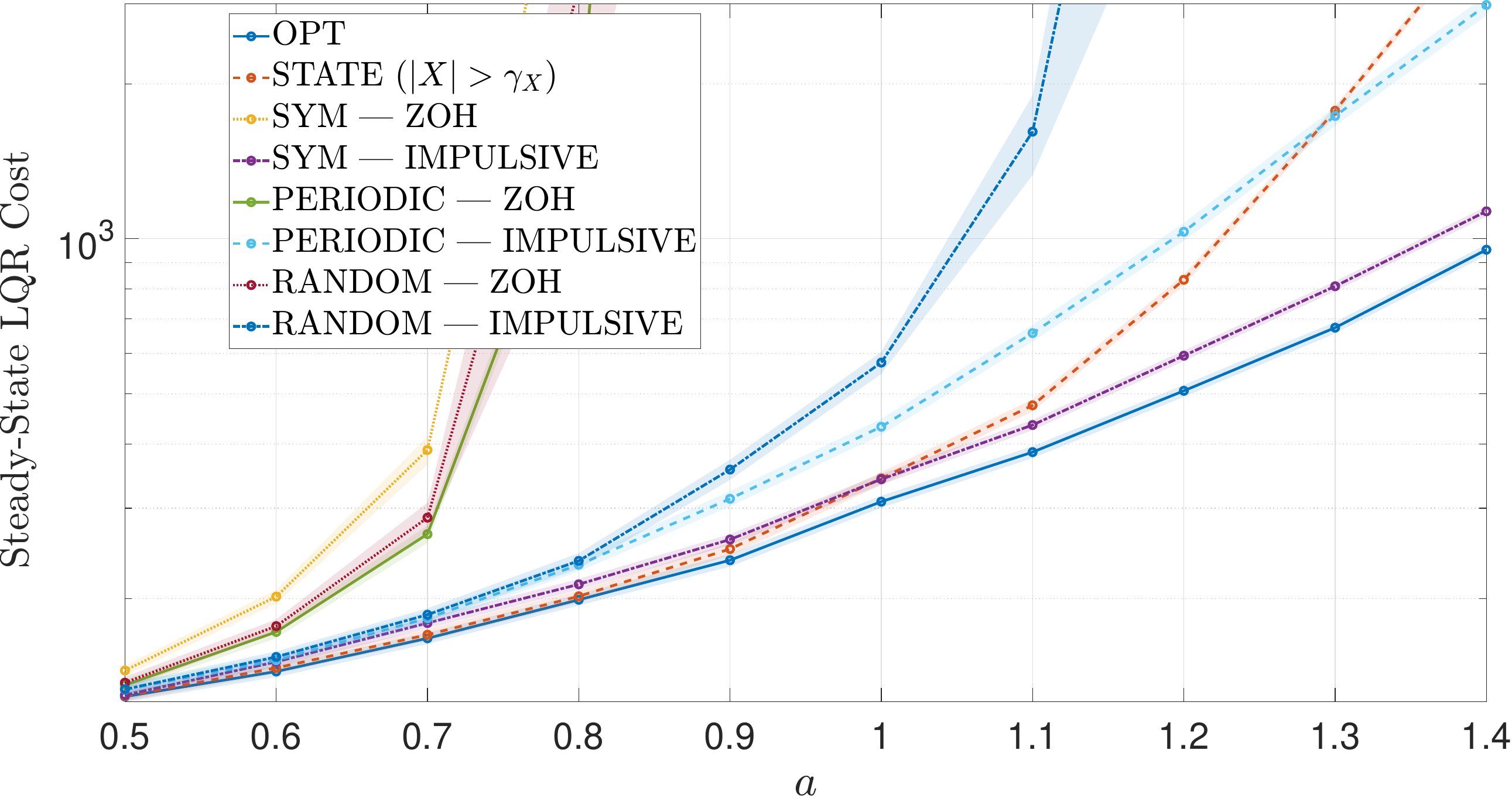}
    \caption{Steady-state LQR cost versus the open-loop gain $a$ under Gaussian disturbances. All policies satisfy $r_s = 0.25$. ZOH-based policies become unstable for $a > 0.9$, while the optimal policy consistently achieves the lowest cost. The symmetric impulsive policy closely tracks the optimal performance across all $a$.}
    \label{fig:LQR_vs_a}
\end{figure}

\begin{figure}[!t]
    \centering
    \includegraphics[width=0.8\linewidth]{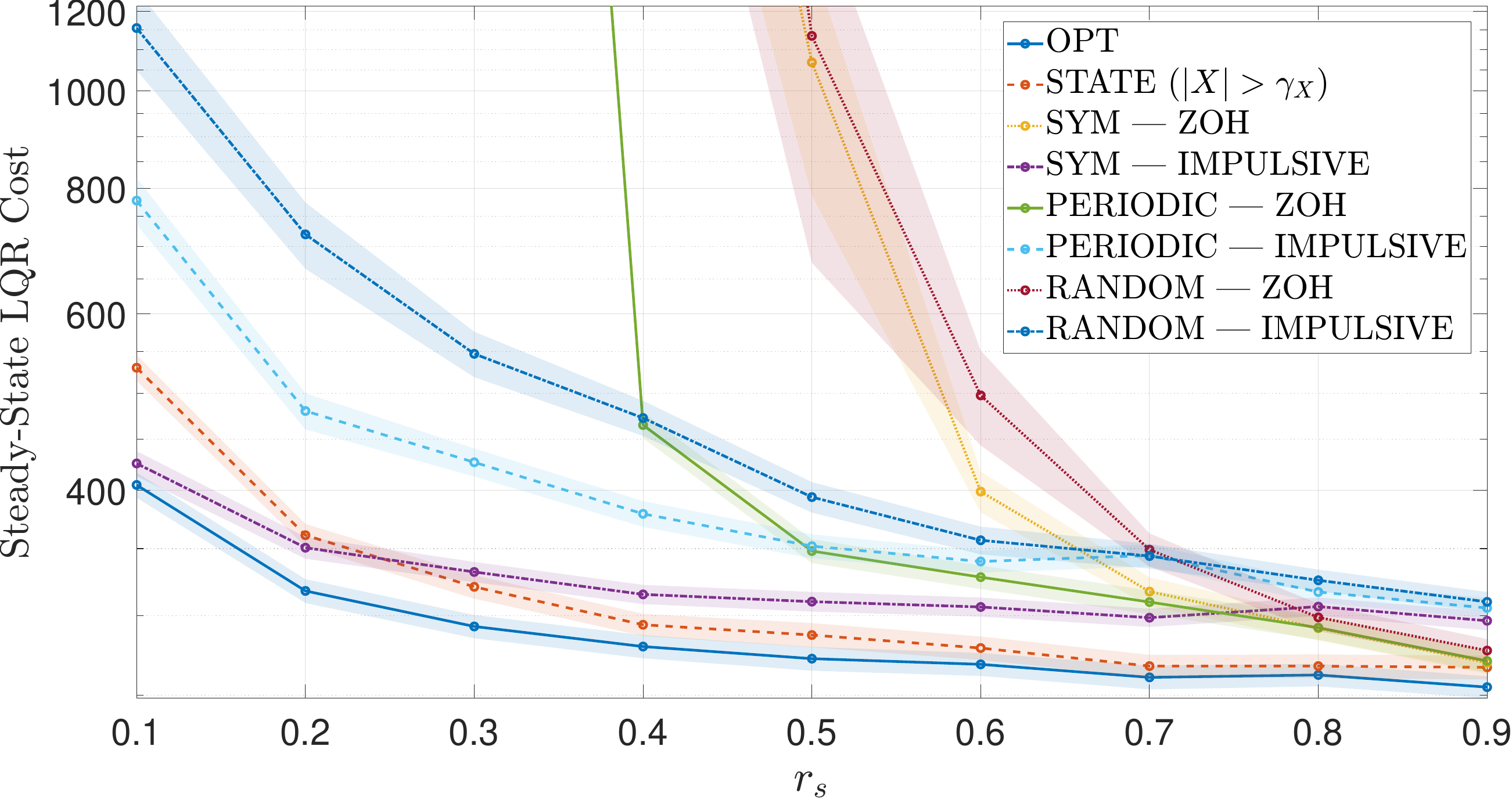}
    \caption{Steady-state LQR cost versus the target communication rate $r_s$ for $a = 1$ under Gaussian disturbances. ZOH policies become unstable for $r_s \lesssim 0.4$, while the optimal policy achieves the lowest cost across all rates. As $r_s \to 1$, all policies converge to similar performance.}
    \label{fig:LQR_vs_p_Gauss}
\end{figure}

As benchmark scheduling strategies, we consider four alternatives: 
\emph{random (Bernoulli)}, \emph{periodic}, \emph{symmetric threshold (SYM)} policies, and state-based policy.

The random (Bernoulli) scheduling policy sets $\{D_k\}$ as an i.i.d.\ Bernoulli sequence with
\[
\mathbb{P}(D_k=1)=r_s,\qquad \mathbb{P}(D_k=0)=1-r_s.
\]
The periodic policy is the ``as periodic as possible'' policy that satisfies the rate constraint for any $0 < r_s < 1$ by distributing scheduling instants as uniformly as possible over time \cite{etcibasi2026freshness}.

\begin{figure}[t]
    \centering
    \includegraphics[width=0.8\linewidth]{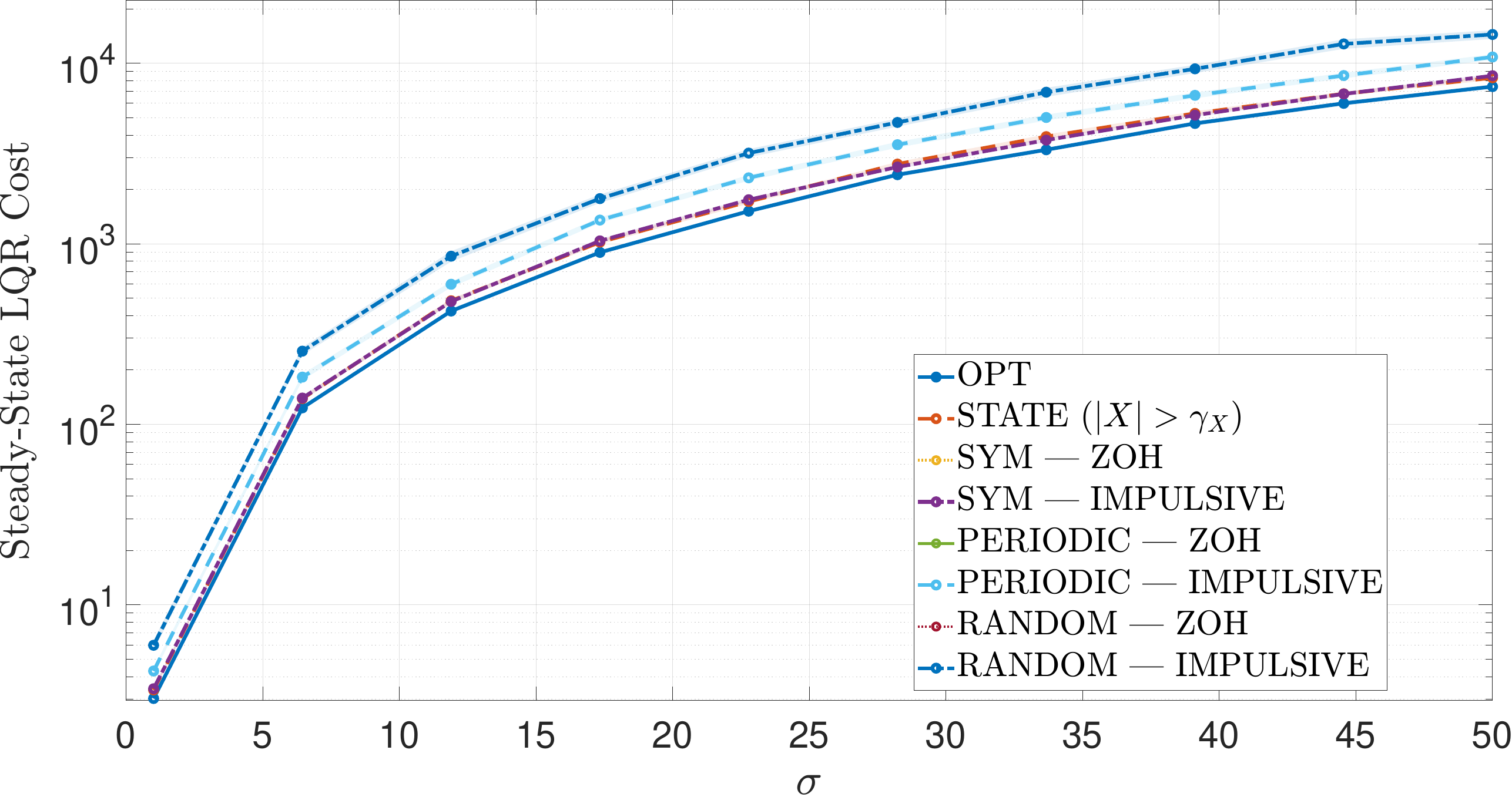}
    \caption{Steady-state LQR cost versus disturbance standard deviation $\sigma_W$ under Gaussian noise with $r_s = 0.25$. ZOH policies incur significantly higher cost and fall outside the plotted range. All visible policies exhibit increasing cost with $\sigma_W$ while maintaining similar relative performance.}
    \label{fig:LQR_vs_sigma}
\end{figure}

As the last scheduling policy, we consider the symmetric scheduling policy defined in~\eqref{eqn:opt_scheduling_pol}, where the threshold is chosen so that the resulting communication rate matches $r_s$. Although the symmetric scheduling policy is optimal when paired with the optimal controller, it becomes suboptimal when combined with ZOH or impulsive controllers, since the controller itself is no longer optimal. Note that in the limiting case $r \to 0$, the optimal controller~\eqref{eqn:optimal_controller} converges to the impulsive controller~\eqref{eqn:imp_controller}. Consequently, as $r \to 0$, the optimal policy~\eqref{eqn:opt_controller_sym},\eqref{eqn:opt_scheduling_pol} coincides with the impulsive-controller and symmetric-scheduling combination.

All of the sub-optimal scheduling policies defined above satisfy the symmetry condition in~\eqref{defn:sym_pol}. 
To broaden the range of policies under consideration, we also introduce a non-symmetric scheduling rule, referred to as the \emph{state-based policy}, defined as
\begin{equation}
    D_k = \mathds{1}_{\{|X_k| > \gamma\}}.
    \label{eqn:State_swithing_policy}
\end{equation}
Although this rule has a symmetric threshold structure in the system state $X_k$, it does not preserve symmetry at the controller level. In particular, it induces a non-zero conditional expectation of the accumulated noise:
\begin{equation*}
    \mathbb{E}\!\left[ S_m \mid I^\mathcal{C}_{t_k + m} \right] \neq 0,
\end{equation*}
since the scheduling decision now depends on the most recent state observation $X_{t_k}$, thereby breaking the symmetry.

\begin{figure}[t]
    \centering
    \includegraphics[width=0.8\linewidth]{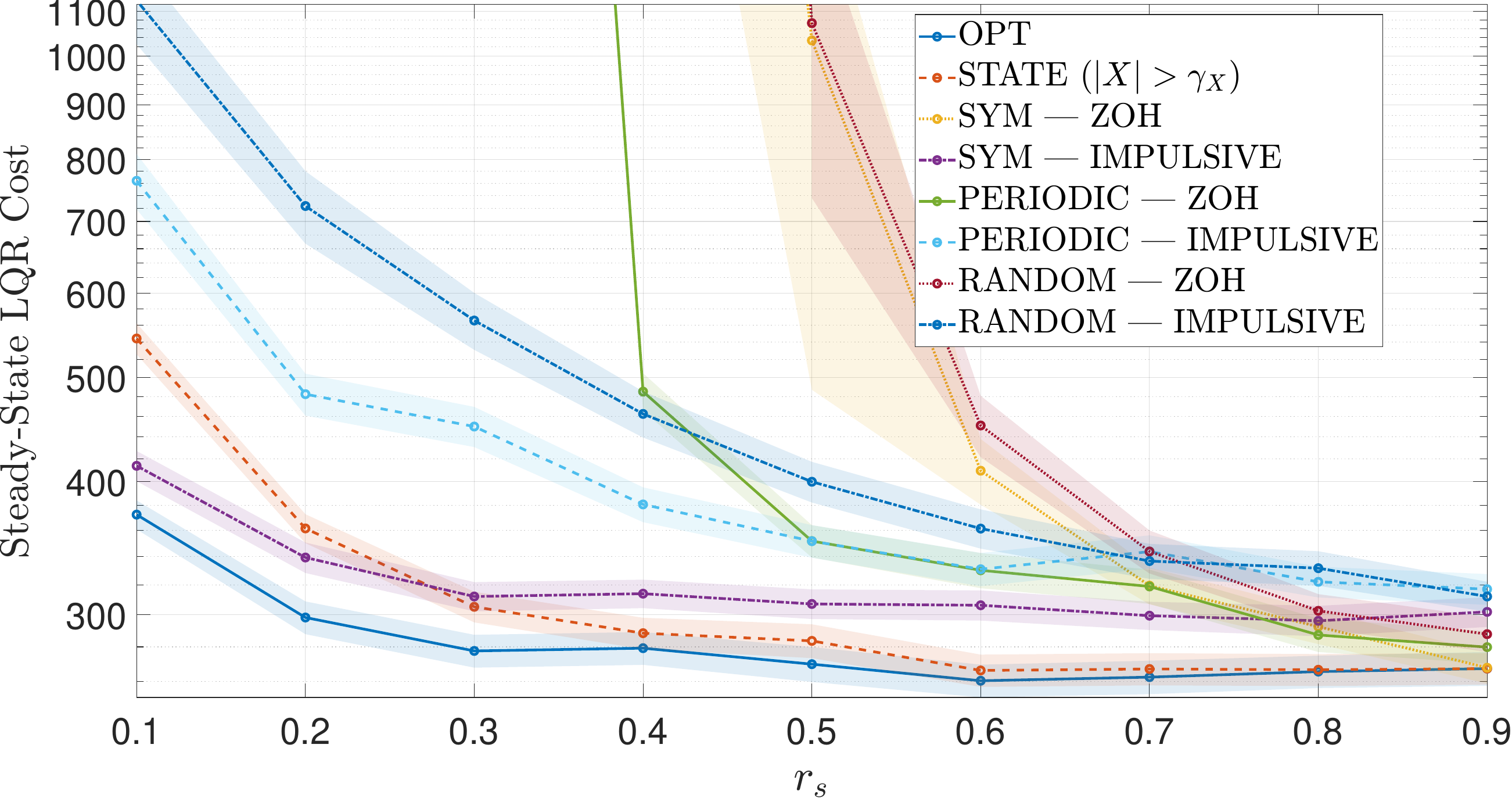}
    \caption{Steady-state LQR cost versus $r_s$ for $a = 1$ under Laplace disturbances. The behavior closely matches the Gaussian case: ZOH policies become unstable for $r_s < 0.4$, and the optimal policy achieves the lowest cost across all rates.}
    \label{fig:LQR_vs_p_Laplace}
\end{figure}

A key consequence is that the controller no longer retains the simple structure in~\eqref{eqn:opt_controller_sym}. 
Instead, the conditional state estimate $\hat{X}_k$ incorporates bias terms that depend explicitly on the noise distribution and the scheduling decisions, as reflected in the modified estimation error dynamics~\eqref{eqn:est_error_dynamics}. 
More precisely, the optimal controller takes the form
\begin{equation}
    U^{\mathrm{State}}_{t_k+m}
    =
    -L \left(
        \xi_m
        +
        \sum_{i=0}^{m-1}
        a^{m-1-i} C_{i,m}
    \right),
    \label{eqn:State_controller}
\end{equation}
where
\begin{align*}
    C_{i,m} &\triangleq \mathbb{E}[W_{t_k+i} \mid \mathcal{A}_m], \\
    \xi_1 &\triangleq a X_{t_k} - bL \hat{X}_{t_k}, \\
    \xi_m &\triangleq (a-bL)\xi_{m-1}
              - bL \sum_{i=0}^{m-2} a^{m-2-i} C_{i,m-1}, \qquad m \geq 2
\end{align*}
and
\begin{align*}
    \mathcal{A}_m &= \bigcap_{j=1}^m \left\{\left|\xi_j + \sum_{i=0}^{j-1} a^{j-1-i} W_{t_k+i}\right| \le \gamma \right\}.
\end{align*}
The derivation is provided in Appendix~\ref{appendix:G}. 
As evident from the above expressions, even a seemingly simple deviation from symmetry significantly complicates the controller structure.

Finally, we evaluate performance under three disturbance models: Gaussian, uniform, and double-exponential (Laplace), each normalized to have variance $\sigma_W^2$. 
The Laplace distribution is given by
\begin{equation}
    f_W(w)
    =
    \frac{1}{2\lambda} e^{-\frac{|w|}{\lambda}},
    \qquad
    \lambda = \frac{\sigma_W}{\sqrt{2}}.
\end{equation}

\begin{figure}[t]
    \centering
    \includegraphics[width=0.8\linewidth]{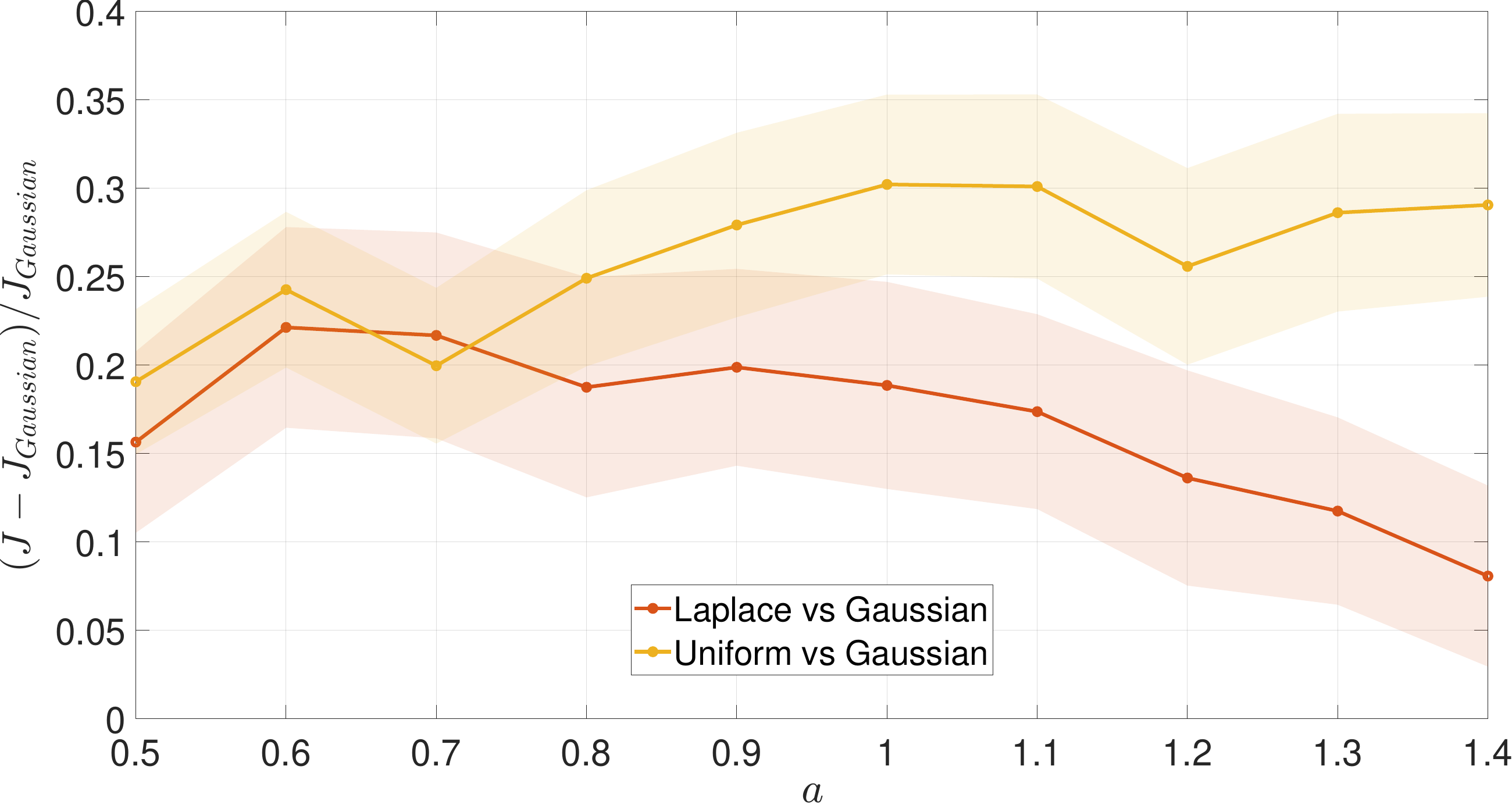}
    \caption{Normalized steady-state cost difference relative to Gaussian disturbances versus $a$ for the optimal policy. The impact of the noise distribution decreases as $a$ increases.}
    \label{fig:Diff_wrt_Gauss_vs_a_OPT}
\end{figure}

We consider a scalar LTI system with parameters $a = 1$, $b = 1$, $q = 1$, $r = 1$, disturbance variance $\sigma_W = 10$, and delay $\tau = 1$, unless stated otherwise. The communication constraint is imposed through a maximum allowable communication rate, denoted by $r_s$. Since the communication rate induced by threshold-based policies depends implicitly on the chosen threshold, we tune the thresholds of all policies so that their empirical communication rates match $r_s$ as closely as possible. This ensures a fair comparison across different policies under a common communication budget.

Our simulation study consists of three main experiments. First, we evaluate the transient performance by plotting the running-average LQR cost as a function of time for a fixed parameter setting and noise distribution. Second, we investigate the sensitivity of the system to individual parameters by varying one parameter at a time while keeping the others fixed. Finally, we examine the impact of the noise distribution by comparing the steady-state cost under different noise types against the Gaussian baseline, thereby quantifying the effect of distributional mismatch on control performance.

In Fig.~\ref{fig:Single_LQR}, we plot the running-average LQR cost as a function of time for the nominal parameter setting with $r_s = 0.4$ under Gaussian disturbances. All considered policy combinations are shown, as indicated in the legend. The empirical communication rate achieved by each policy is also reported in the legend as ``rate.'' We first observe that certain policies, namely Random-ZOH and Sym-ZOH, lead to unstable behavior, as evidenced by the divergence of their cost. Among the stable policies, periodic-ZOH yields the worst performance, indicating that ZOH implementations are generally suboptimal under this setting. As expected, the proposed optimal policy achieves the lowest cost. The state-based policy and the symmetric impulsive (SYM-IMP) policy exhibit the next best performance. This aligns with the theoretical result that the symmetric impulsive policy is optimal in the limit $r \to 0$, with the observed performance gap attributable to the finite value $r = 1$.

\begin{figure}[t]
    \centering
    \includegraphics[width=0.8\linewidth]{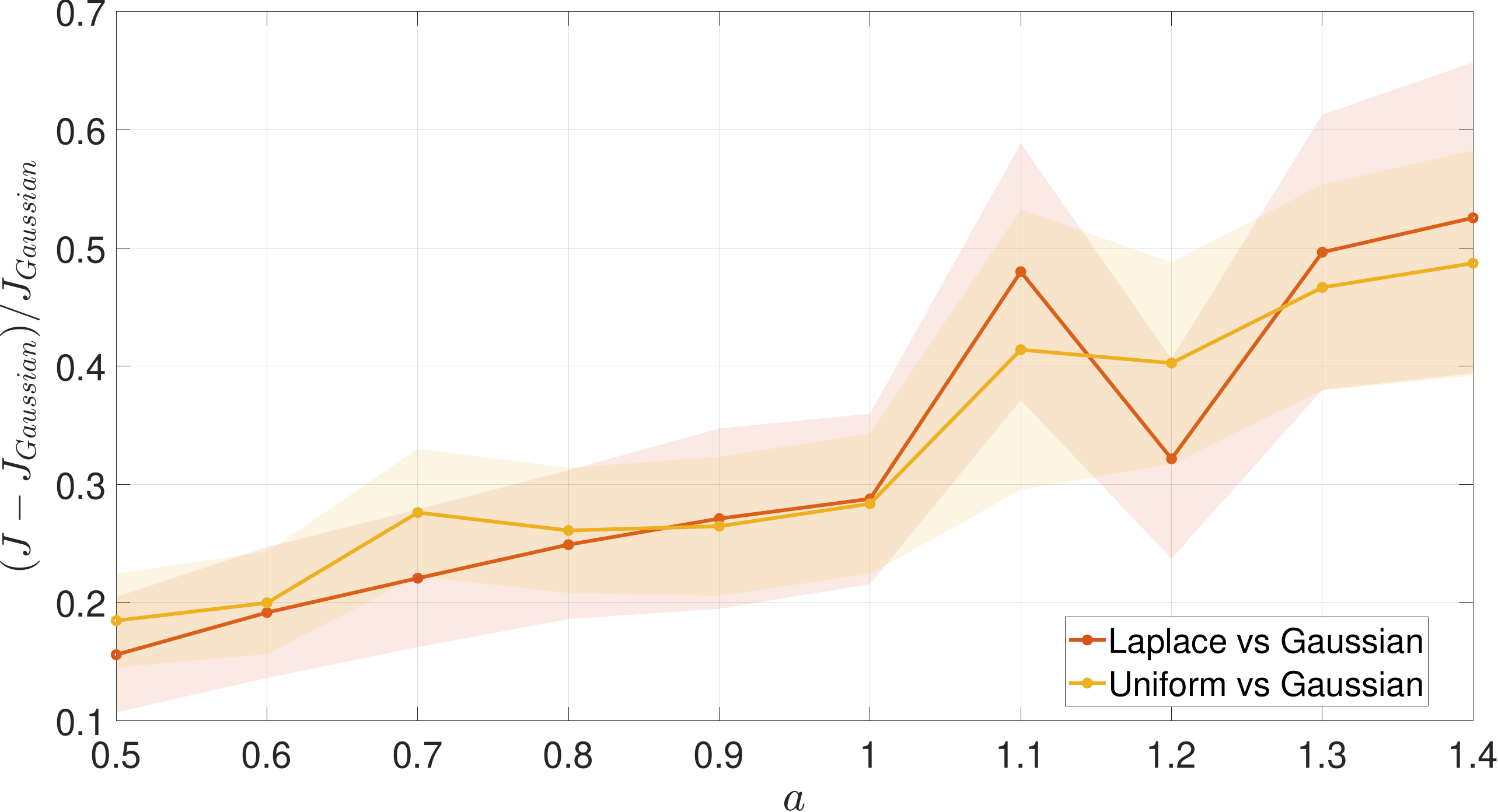}
    \caption{Normalized steady-state cost difference relative to Gaussian disturbances versus $a$ for the periodic impulsive policy. The performance is largely insensitive to the noise distribution due to the independence of scheduling decisions from system dynamics.}
    \label{fig:Diff_wrt_Gauss_vs_a_PER_IMP}
\end{figure}

We next examine how the relative performance of these policies varies with the open-loop gain $a$, the communication rate $r_s$, and the disturbance variance $\sigma_W$.
In Fig.~\ref{fig:LQR_vs_a}, we plot the steady-state LQR cost as a function of the open-loop gain $a$ under Gaussian disturbances with $r_s = 0.25$. The steady-state cost is computed as the average cost over the final 20 steps of a horizon of length 100. Shaded regions indicate $95\%$ confidence intervals. Consistent with the transient results, policies based on ZOH become unstable as the open-loop system becomes more unstable. In particular, all ZOH policies diverge for $a > 0.9$. Additionally, the random impulsive policy becomes unstable for sufficiently large $a$ (approximately $a > 1.2$).

The proposed optimal policy achieves the lowest cost across all values of $a$. Among suboptimal policies, the relative performance depends on the system dynamics: the state-based policy performs well for stable systems ($a < 1$) but degrades and eventually becomes unstable as $a$ increases. In contrast, the symmetric impulsive (SYM-IMP) policy remains stable and closely tracks the optimal policy across the entire range of $a$, consistent with its near-optimality for finite $r$ and its optimality in the limit $r \to 0$.

In Fig.~\ref{fig:LQR_vs_p_Gauss}, we plot the steady-state LQR cost as a function of the target communication rate $r_s$ for $a = 1$ under Gaussian disturbances. The optimal policy consistently achieves the lowest cost across all communication rates. As the communication constraint becomes more stringent (i.e., as $r_s$ decreases), ZOH policies rapidly degrade and become unstable. In particular, all ZOH policies diverge for $r_s < 0.4$, consistent with the observations in Fig.~\ref{fig:Single_LQR}. While the confidence intervals appear larger in this figure, this is primarily due to the scaling of the vertical axis. Finally, as $r_s \to 1$, all policies exhibit similar performance, reflecting the diminishing impact of communication constraints when transmissions are nearly always allowed.

In Fig.~\ref{fig:LQR_vs_sigma}, we plot the steady-state LQR cost as a function of the disturbance standard deviation $\sigma_W$ for $r_s = 0.25$ under Gaussian noise. For this value of $r_s$, ZOH-based policies are unstable (see Fig.~\ref{fig:LQR_vs_p_Gauss}) and yield significantly larger costs. As a result, their curves lie outside the displayed range of the vertical axis and are not visible in the plot. As expected, the steady-state cost increases monotonically with $\sigma_W$, reflecting the growing impact of disturbances. However, the growth rate diminishes for large $\sigma_W$, leading to a reduced relative gap between policies. Notably, the performance ordering remains largely unchanged across all noise levels, with the optimal and symmetric impulsive policies consistently outperforming the others. We next investigate the effect of non-Gaussian disturbance distributions.

In Fig.~\ref{fig:LQR_vs_p_Laplace}, we plot the steady-state LQR cost as a function of the target communication rate $r_s$ for $a = 1$ under Laplace disturbances. Comparing with Fig.~\ref{fig:LQR_vs_p_Gauss}, we observe qualitatively similar behavior across all policies. In particular, ZOH-based policies become unstable for $r_s < 0.4$, and the state-based policy consistently provides the second-best performance for moderate to large communication rates ($r_s > 0.3$). These observations are consistent with our theoretical results, which establish the optimality of the proposed policy for all symmetric, zero-mean disturbance distributions. While the optimal policy remains unchanged, it is of interest to examine how the relative performance of suboptimal policies varies across different noise distributions. We explore this comparison next.

In Fig.~\ref{fig:Diff_wrt_Gauss_vs_a_OPT}, we plot the normalized difference in steady-state LQR cost between non-Gaussian and Gaussian disturbances as a function of the open-loop gain $a$ for the optimal policy. The results show that the noise distribution has a noticeable impact on performance. In particular, the gap between Laplace and Gaussian costs decreases as $a$ increases, indicating that the effect of the distribution diminishes for more unstable systems. In contrast, Fig.~\ref{fig:Diff_wrt_Gauss_vs_a_PER_IMP} shows the same comparison for the periodic impulsive policy. Here, the performance under different noise distributions remains nearly identical across all values of $a$. This can be attributed to the fact that, under periodic scheduling, the transmission decisions are independent of the system state and noise realizations. As a result, the interaction between the noise distribution and the scheduling mechanism is eliminated, leading to distribution-insensitive performance.

To further illustrate this effect, we plot the normalized difference with respect to the Gaussian case as a function of $r_s$ in Figs.~\ref{fig:Diff_wrt_Gauss_vs_p_OPT} and~\ref{fig:Diff_wrt_Gauss_vs_p_PER_IMP}. In Fig.~\ref{fig:Diff_wrt_Gauss_vs_p_OPT}, we plot the normalized steady-state cost difference relative to the Gaussian case as a function of $r_s$ for the optimal policy. As in the previous results, the performance depends on the underlying noise distribution, indicating that the interaction between the scheduling decisions and the disturbance remains significant. In contrast, Fig.~\ref{fig:Diff_wrt_Gauss_vs_p_PER_IMP} shows that for the periodic impulsive policy, the normalized cost difference is nearly identical for both Laplace and uniform disturbances across all values of $r_s$. This suggests that the performance of such policies is largely insensitive to the specific noise distribution. This behavior can be explained by the fact that periodic (self-triggered) scheduling policies are independent of the noise realizations. Consequently, for any zero-mean symmetric disturbance distributions with identical variance, the induced closed-loop cost is effectively invariant to the choice of distribution.

\begin{figure}[t]
    \centering
    \includegraphics[width=0.8\linewidth]{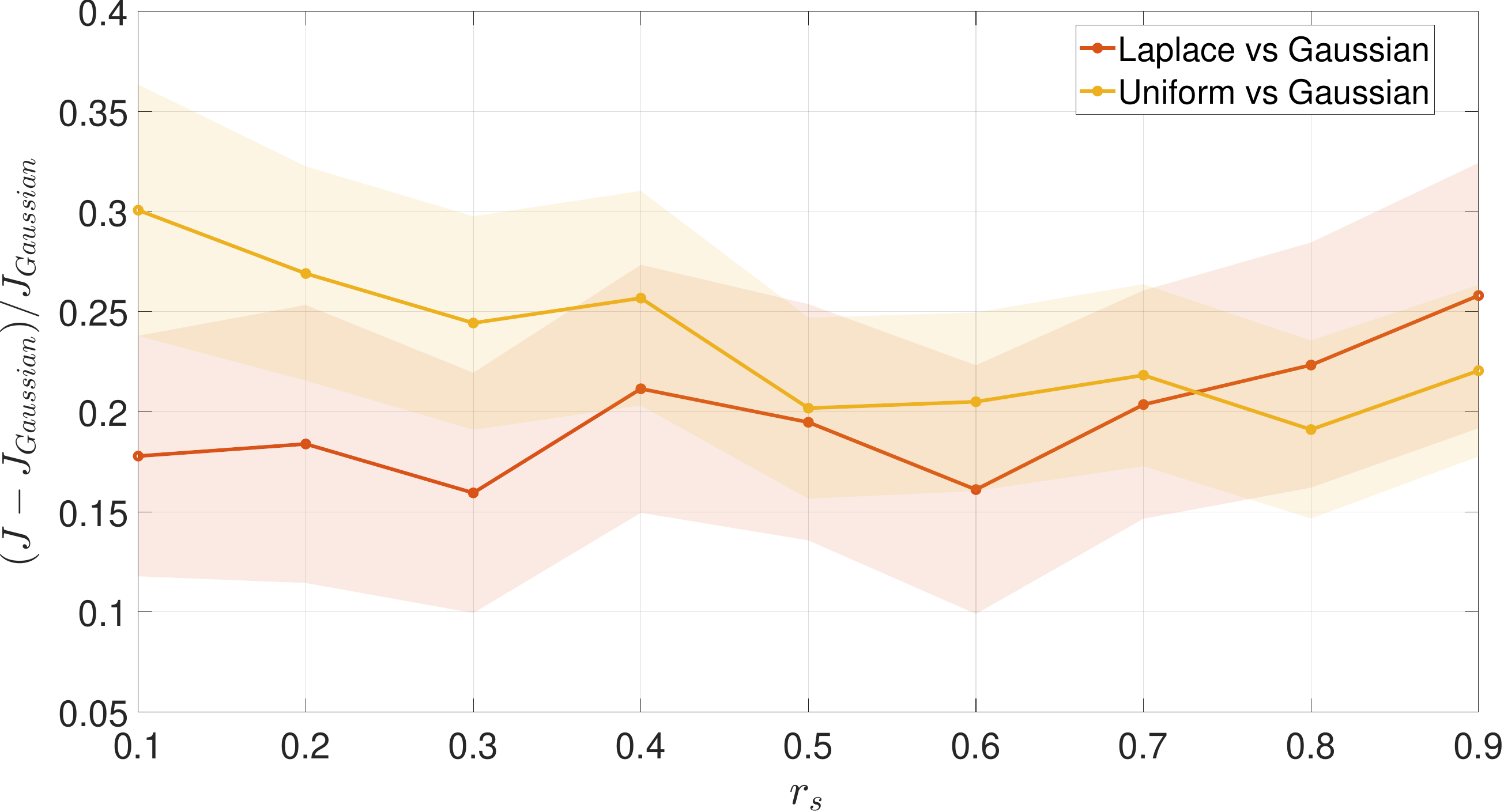}
    \caption{Normalized steady-state cost difference relative to Gaussian disturbances versus $r_s$ for the optimal policy. The performance varies with the noise distribution across all communication rates.}
    \label{fig:Diff_wrt_Gauss_vs_p_OPT}
\end{figure}

\begin{figure}[t]
    \centering
    \includegraphics[width=0.8\linewidth]{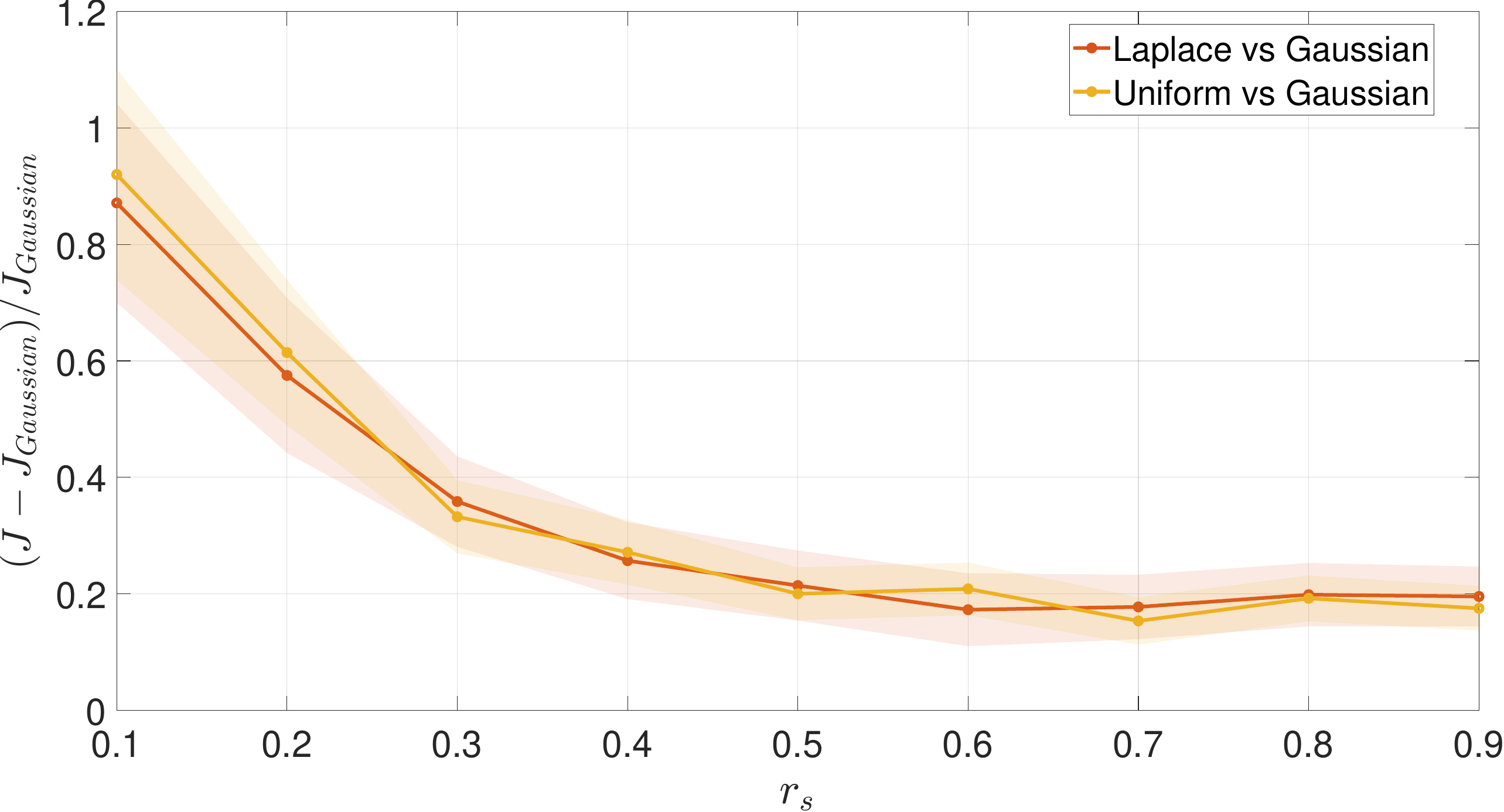}
    \caption{Normalized steady-state cost difference relative to Gaussian disturbances versus $r_s$ for the periodic impulsive policy. The performance is nearly identical across noise distributions, indicating distribution insensitivity.}
    \label{fig:Diff_wrt_Gauss_vs_p_PER_IMP}
\end{figure}

\section{Conclusion}

This paper studied the constrained joint design of scheduling and control policies for scalar networked control systems under an average communication-rate constraint. The main difficulty arises from the dual effect: scheduling decisions influence the controller’s information pattern and hence the estimation process, which in general prevents the direct application of the classical separation principle. To address this issue, we carefully analyzed the underlying information structure and identified the class of symmetric policies, under which the accumulated disturbance remains conditionally zero-mean at the controller.

Our main result established that this class is not merely a tractable restriction, but is in fact optimal for the finite-horizon communication-constrained LQR problem. As a consequence, separation holds under the optimal policy: the optimal controller retains the standard linear feedback form through the conditional state estimate, while the optimal scheduling mechanism is characterized by a symmetric threshold structure. In particular, the scheduling decision can be expressed in terms of the accumulated disturbance, or equivalently the estimation error under symmetry.

These results clarify a structural assumption that has been used, often implicitly, throughout the event-triggered and communication-constrained control literature. They also show that, under the optimal symmetric policy, the absence of a transmission does not introduce an additional bias into the controller’s estimate, thereby recovering a tractable and interpretable closed-loop structure.

\section*{Acknowledgments}
The author thanks Halit Bugra Tulay and Artun Sel for valuable discussions.

\appendix
\section{Proof of Theorem \ref{theorem:optimal control}} \label{Appendix_A}
We will use induction in this proof. The main steps follow from the Bertsekas's book \cite[Chapter 5.2]{bertsekas2012dynamic}. However,  imperfect information and delay-free feedback loop have been assumed in~\cite{bertsekas2012dynamic}, and not partial observations and fixed delay. We will show that this proof can be extended to partial observation problem and fixed delay. Note that the fixed $\tau$ delay only affects the controller's information $I^\mathcal{C}_{k}$. Hence, we need to be careful when using $I^\mathcal{C}_{k}$.

Let the cost-to-go function be
\begin{equation}
    J_k^\star(I^\mathcal{C}_{k}) = \min_{u_k} \mathbb{E} \left[ q X_k^2 + r U_k^2 + J_{k+1}(I^\mathcal{C}_{k+1}) \,\middle|\, I^\mathcal{C}_{k},\, U_k = u_k \right]
\end{equation}

Let's define a terminal cost \(q\) for \(X_N\), hence
\begin{equation}
    J_{N-1}^\star(I^\mathcal{C}_{N-1}) = \min_{u_{N-1}} \mathbb{E} \left[ q X_{N-1}^2 + r U_{N-1}^2 + qX_N^2 \,\middle|\, I^\mathcal{C}_{N-1},\, U_{N-1} = u_{N-1} \right] \label{eqn:J_N-1}
\end{equation}

Note that, no matter what \(D_{N-1}\) is, the system state evolution is \(X_{N} = aX_{N-1} + bU_{N-1} + W_{N-1}\), since \(X_{N-1}\) is independent of $D_{N-1}$ and $W_{N-1}$.

Hence, by using the system state evolution \eqref{eqn:system_Model}, \eqref{eqn:J_N-1} can be written as
\begin{align*}
    J_{N-1}^\star(I^\mathcal{C}_{N-1}) & = \min_{u_{N-1}} \Big\{
    q(1 + a^2) \, \mathbb{E}[X_{N-1}^2 \mid I^\mathcal{C}_{N-1}]
    + (r + b^2 q) u_{N-1}^2 \\
    & \qquad + 2abq \, \mathbb{E}[X_{N-1} \mid I^\mathcal{C}_{N-1}] \, u_{N-1} + q \, \mathbb{E}[W_{N-1}^2 \mid I^\mathcal{C}_{N-1}] \\
    & \qquad + 2bq \, \mathbb{E}[W_{N-1} \mid I^\mathcal{C}_{N-1}] \, u_{N-1} \Big\} \\
    & \overset{(a)}{=} \min_{u_{N-1}} \Big\{(r + b^2 q) u_{N-1}^2 + 2abq \, \mathbb{E}[X_{N-1} \mid I^\mathcal{C}_{N-1}] \, u_{N-1}\Big\} \\
    & \qquad + q(1 + a^2) \, \mathbb{E}[X_{N-1}^2 \mid I^\mathcal{C}_{N-1}] + q \sigma_W^2 \\
    & \overset{(b)}{=} q(1 + a^2) \, \mathbb{E}[X_{N-1}^2 \mid I^\mathcal{C}_{N-1}] - \frac{(abq)^2}{r + b^2 q} \left( \mathbb{E}[X_{N-1} \mid I^\mathcal{C}_{N-1}] \right)^2 + q \sigma_W^2 \\[1ex]
    & \overset{(c)}{=} \mathbb{E} \left[ q(1 + a^2) X_{N-1}^2 - \frac{(abq)^2}{r + b^2 q} \left( \mathbb{E}[X_{N-1} \mid I^\mathcal{C}_{N-1}] \right)^2 + q \sigma_W^2 \,\middle|\, I^\mathcal{C}_{N-1}
    \right] \numberthis \label{eqn:J_N-1_2}
\end{align*}
where in the Step (a), we used the fact that $W_{N-1}$ is independent of $I^\mathcal{C}_{N-1}$ for any $\tau$. In Step (b), we solved and used the optimal controller as follows
\begin{align*}
    u_{N-1}^\star & = \arg\min_{u_{N-1}} \left\{
    (r + b^2 q) u_{N-1}^2
    + 2abq \, \mathbb{E}[X_{N-1} \mid I^\mathcal{C}_{N-1}] \, u_{N-1} \right\}\\
    & = -\frac{abq}{r + b^2 q} \, \mathbb{E}[X_{N-1} \mid I^\mathcal{C}_{N-1}]
\end{align*}
And in Step (c), we expanded the conditional expectation $\mathbb{E}[X_{N-1}^2 \mid I^\mathcal{C}_{N-1}]$. Note that $u_{N-1}^\star$ includes all possible scheduling policies up to $N-1$, i.e. \(\boldsymbol{\pi}=(f_{0},f_{1},\dots, f_{N-1})\in\Pi\).

Now, note that
\begin{align*}
    \mathbb{E} \left[ \left( \mathbb{E}[X_{N-1} \mid I^\mathcal{C}_{N-1}] \right)^2 \,\middle|\, I^\mathcal{C}_{N-1} \right]
    & = \mathbb{E} \bigg[ 2 X_{N-1} \, \mathbb{E}[X_{N-1} \mid I^\mathcal{C}_{N-1}] \\ 
    & \qquad - \Big( \mathbb{E}[X_{N-1} \mid I^\mathcal{C}_{N-1}] \Big)^2 \,\bigg| \, I^\mathcal{C}_{N-1} \bigg]
\end{align*}
Hence, \eqref{eqn:J_N-1_2} can be rewritten as
\begin{align*}
    J_{N-1}^\star(I^\mathcal{C}_{N-1}) & = \mathbb{E} \Bigg[q(1 + a^2) X_{N-1}^2 - \frac{(abq)^2}{r + b^2 q}
    \Big( 2 X_{N-1} \, \mathbb{E}[X_{N-1} \mid I^\mathcal{C}_{N-1}] \\ 
    & \qquad - \left( \mathbb{E}[X_{N-1} \mid I^\mathcal{C}_{N-1}] \right)^2 \Big) + q \sigma_W^2
    \,\Bigg|\, \{I^\mathcal{C}\}_0^N \Bigg] \\[1ex] 
    & = \mathbb{E} \Bigg[\bigg( q(1 + a^2) - \frac{(abq)^2}{r + b^2 q} \bigg) X_{N-1}^2 + \frac{(abq)^2}{r + b^2 q}
    \bigg( X_{N-1}^2 - 2 X_{N-1} \, \mathbb{E}[X_{N-1} \mid I^\mathcal{C}_{N-1}] \\ 
    & \qquad + \left( \mathbb{E}[X_{N-1} \mid I^\mathcal{C}_{N-1}] \right)^2 \bigg) + q \sigma_W^2
    \,\Bigg|\, I^\mathcal{C}_{N-1} \Bigg] \\
    & = P_{N-1} \mathbb{E}\left[ X_{N-1}^2 \mid I^\mathcal{C}_{N-1} \right] \\
    & \qquad + \frac{(abP_N)^2}{r + b^2 P_N} \, \mathbb{E}\left[\left( X_{N-1} - \mathbb{E}[X_{N-1} \mid I^\mathcal{C}_{N-1}] \right)^2
    \,\middle|\, I^\mathcal{C}_{N-1}\right] + q \sigma_W^2
\end{align*}
where
\[
 P_{N-1} = \left(q + a^2P_N - \frac{(abP_N)^2}{r + b^2 P_N}\right), \qquad P_N = q.
\]
Also note that
\begin{align*}
    \mathbb{E} \bigg[ \Big( X_{N-1} & - \mathbb{E}[X_{N-1} \mid I^\mathcal{C}_{N-1}] \Big)^2 \,\bigg|\, I^\mathcal{C}_{N-1} \bigg] \\
    & = \mathbb{E} \left[\left(a X_{N-2} + b U_{N-2} + W_{N-2} - \mathbb{E}\left[ a X_{N-2} + b U_{N-2} + W_{N-2} \mid I^\mathcal{C}_{N-1} \right]\right)^2 \middle|\, I^\mathcal{C}_{N-1} \right] \\
    & \overset{(a)}{=} \mathbb{E} \Bigg[ \bigg(a \Big( X_{N-2} - a \mathbb{E}[X_{N-2} \mid I^\mathcal{C}_{N-1}] \Big) + \Big( W_{N-2} - \mathbb{E}[W_{N-2} \mid I^\mathcal{C}_{N-1}] \Big)\bigg)^2 \Bigg|\, I^\mathcal{C}_{N-1} \Bigg]
\end{align*}
where Step (a) holds because $U_{N-2} \in I^\mathcal{C}_{N-1}$, independent from $\tau$.

Now, let us use induction. Assume
\begin{align*}
    J_{k+1}^\star(I^\mathcal{C}_{k+1}) & = P_{k+1} \, \mathbb{E}[X_{k+1}^2 \mid I^\mathcal{C}_{k+1}] \\
    & \qquad + \sum_{j=k+1}^{N-1} \left(\frac{(ab\, P_{j+1})^2}{r + b^2 P_{j+1}} \mathbb{E}\left[\left( X_{j} - \mathbb{E}[X_{j} \mid I^\mathcal{C}_{j}] \right)^2\,\middle|\, I^\mathcal{C}_{k+1}\right] + P_{j+1} \sigma_W^2 \right) \numberthis \label{eqn:J_k+1}
\end{align*}
Hence,
\begin{align*}
    J_k^\star(I^\mathcal{C}_{k}) & = \min_{u_k} \mathbb{E} \left[q X_k^2 + r U_k^2 + J_{k+1}(I^\mathcal{C}_{k+1}) \,\middle|\, I^\mathcal{C}_{k},\, U_k = u_k \right] \\
    & = \min_{u_k} \Bigg\{ r\,u_k^2 + P_{k+1}\,\mathbb{E}\!\left[ \mathbb{E}\!\left[ X_{k+1}^2 \mid I^\mathcal{C}_{k+1} \right] \,\middle|\, I^\mathcal{C}_{k},\, U_k = u_k \right]  \Bigg\} \\
    & \qquad + q\,\mathbb{E}\!\left[ X_k^2 \mid I^\mathcal{C}_{k},\, U_k = u_k \right] + \sum_{j=k+1}^{N-1} \Big( \frac{(ab\,P_{j+1})^2}{r + b^2 P_{j+1}} \\
    & \qquad \mathbb{E}\!\left[ \mathbb{E}\!\left[ \left( X_j - \mathbb{E}[X_j \mid I^\mathcal{C}_{j}] \right)^2 \,\middle|\, I^\mathcal{C}_{k+1} \right] \,\middle|\, I^\mathcal{C}_{k},\, U_k = u_k \right] + P_{j+1}\sigma_W^2 \Big) \numberthis \label{eqn:J_k}
\end{align*}
First, since the system is linear and the information vector $I^\mathcal{C}_{k}$ includes the control signals, the estimation error terms $\left( X_j - \mathbb{E}[X_j \mid I^\mathcal{C}_{j}] \right)$ are independent of $u_k$ for any $\tau$. Furthermore, since $\sigma\left(I^\mathcal{C}_{k}\right) \subset \sigma\left(I^\mathcal{C}_{k+1}\right)$ where $\sigma\left(I^\mathcal{C}_{k}\right)$ is the sigma-field generated by $I^\mathcal{C}_{k}$ and by the tower property,
\begin{align*}
    \mathbb{E} \Bigg[ \mathbb{E} \left[ \left( X_j - \mathbb{E}[X_j \mid I^\mathcal{C}_j] \right)^2 \,\middle|\, I^\mathcal{C}_{k+1} \right] \,\Bigg|\,&  I^\mathcal{C}_{k},\, U_k = u_k \Bigg] \\
    & = \mathbb{E} \left[ \left( X_j - \mathbb{E}[X_j \mid I^\mathcal{C}_j] \right)^2 \,\middle|\, I^\mathcal{C}_{k} \right] \numberthis \label{eqn:J_k_middle1}
\end{align*}
Secondly, since for any $\tau$ the system state evolution is the same for both $D_K=0$ and $D_K=1$, \eqref{eqn:system_Model}, and again by the tower property, we have
\begin{align*}
    \mathbb{E} \bigg[ \mathbb{E}\Big[X_{k+1}^2 \Big| I^\mathcal{C}_{k+1}\Big] \,\bigg|\, I^\mathcal{C}_{k},\,&  U_k = u_k \bigg] = \mathbb{E} \left[ (a X_k + b u_k + W_k)^2 \,\middle|\, I^\mathcal{C}_{k} \right] \\
    & = a^2 \mathbb{E}[X_k^2 \mid I^\mathcal{C}_{k}] + b^2 u_k^2 + \mathbb{E}[W_k^2 \mid I^\mathcal{C}_{k}] + 2ab \mathbb{E}[X_k \mid I^\mathcal{C}_{k}] u_k \\
    & \qquad + 2a \mathbb{E}[X_k W_k \mid I^\mathcal{C}_{k}] + 2b \mathbb{E}[W_k \mid I^\mathcal{C}_{k}] u_k \\
    & \overset{(a)}{=} a^2 \mathbb{E}[X_k^2 \mid I^\mathcal{C}_{k}] + b^2 u_k^2 + \sigma_W^2 + 2ab \mathbb{E}[X_k \mid I^\mathcal{C}_{k}] u_k \numberthis \label{eqn:J_k_middle2}
\end{align*}
where in Step (a) we used the fact that $I^\mathcal{C}_{k}$ and $X_k$ are independent of $W_k$ for any $\tau$.
By substituting \eqref{eqn:J_k_middle1} and \eqref{eqn:J_k_middle2} into \eqref{eqn:J_k}, we obtain
\begin{align*}
    J_k^\star & (I^\mathcal{C}_{k}) = \min_{u_k} \Bigg\{ r u_k^2 + P_{k+1} \left( a^2 \mathbb{E}[X_k^2 \mid I^\mathcal{C}_{k}] + b^2 u_k^2 + \sigma_W^2 + 2ab \mathbb{E}[X_k \mid I^\mathcal{C}_{k}] u_k \right) \Bigg\} \\
    & \quad + q \mathbb{E}[X_k^2 \mid I^\mathcal{C}_{k}] + \sum_{j = k+1}^{N-1} \left( \frac{(ab \, P_{j+1})^2}{r + b^2 P_{j+1}}  \mathbb{E} \left[ \left( X_j - \mathbb{E}[X_j \mid I^\mathcal{C}_j] \right)^2 \,\middle|\, I^\mathcal{C}_{k} \right] + P_{j+1} \sigma_W^2 \right)
\end{align*}
So, the optimal controller is given as
\begin{equation}
    u_k^\star = -\frac{ab \, P_{k+1}}{r + b^2 P_{k+1}} \, \mathbb{E}[X_k \mid I^\mathcal{C}_{k}] \label{eqn:u_k}
\end{equation}
We then find the optimal cost-to-go function as
\begin{align*}
    J_k^\star(I^\mathcal{C}_{k}) & = \frac{(ab \, P_{k+1})^2}{(r + b^2 P_{k+1})^2} \left( \mathbb{E}[X_k \mid I^\mathcal{C}_{k}] \right)^2 - 2 \frac{(ab \, P_{k+1})^2}{(r + b^2 P_{k+1})} \left( \mathbb{E}[X_k \mid I^\mathcal{C}_{k}] \right)^2 + P_{k+1} \sigma_W^2 \\
    & \qquad + P_{k+1} a^2 \mathbb{E}[X_k^2 \mid I^\mathcal{C}_{k}] + q \mathbb{E}[X_k^2 \mid I^\mathcal{C}_{k}] \\
    & \qquad + \sum_{j = k+1}^{N-1} \left( \frac{(ab \, P_{j+1})^2}{r + b^2 P_{j+1}} \mathbb{E} \left[ \left( X_j - \mathbb{E}[X_j \mid I^\mathcal{C}_j] \right)^2 \,\middle|\, I^\mathcal{C}_{k} \right] + P_{j+1} \sigma_W^2 \right) \\
    & = \left(q + a^2 P_{k+1} \right) \mathbb{E}[X_k^2 \mid I^\mathcal{C}_{k}] - \frac{(ab \, P_{k+1})^2}{r + b^2 P_{k+1}} \left( \mathbb{E}[X_k \mid I^\mathcal{C}_{k}] \right)^2 + P_{k+1} \sigma_W^2 \\
    & \qquad + \sum_{j = k+1}^{N-1} \left( \frac{(ab \, P_{j+1})^2}{r + b^2 P_{j+1}} \mathbb{E} \left[ \left( X_j - \mathbb{E}[X_j \mid I^\mathcal{C}_j] \right)^2 \,\middle|\, I^\mathcal{C}_{k} \right] + P_{j+1} \sigma_W^2 \right) \\
    & \overset{(a)}{=} \mathbb{E} \left[ \left(q + a^2 P_{k+1} \right) X_k^2 - \frac{(ab \, P_{k+1})^2}{r + b^2 P_{k+1}} \left( 2 X_k \mathbb{E}[X_k \mid I^\mathcal{C}_{k}] - \left( \mathbb{E}[X_k \mid I^\mathcal{C}_{k}] \right)^2 \right) \,\middle|\, I^\mathcal{C}_{k} \right]  \\
    & \quad + P_{k+1} \sigma_W^2 + \sum_{j = k+1}^{N-1} \left( \frac{(ab \, P_{j+1})^2}{r + b^2 P_{j+1}} \mathbb{E} \left[ \left( X_j - \mathbb{E}[X_j \mid I^\mathcal{C}_j] \right)^2 \,\middle|\, I^\mathcal{C}_{k} \right] + P_{j+1} \sigma_W^2 \right) \\
    & = \mathbb{E} \bigg[ \left( q + a^2 P_{k+1} - \frac{(ab \, P_{k+1})^2}{r + b^2 P_{k+1}} \right) X_k^2 \\
    & \qquad + \frac{(ab \, P_{k+1})^2}{r + b^2 P_{k+1}} \left( X_k^2 - 2 X_k \mathbb{E}[X_k \mid I^\mathcal{C}_{k}] + \left( \mathbb{E}[X_k \mid I^\mathcal{C}_{k}] \right)^2 \right) \,\bigg|\, I^\mathcal{C}_{k} \bigg]  \\
    & \quad + P_{k+1} \sigma_W^2 + \sum_{j = k+1}^{N-1} \left( \frac{(ab \, P_{j+1})^2}{r + b^2 P_{j+1}} \mathbb{E} \left[ \left( X_j - \mathbb{E}[X_j \mid I^\mathcal{C}_j] \right)^2 \,\middle|\, I^\mathcal{C}_{k} \right] + P_{j+1} \sigma_W^2 \right) \\
    & = \left( q + a^2 P_{k+1} - \frac{(ab \, P_{k+1})^2}{r + b^2 P_{k+1}} \right) \mathbb{E}[X_k^2 \mid I^\mathcal{C}_{k}] \\
    & \qquad + \frac{(ab \, P_{k+1})^2}{r + b^2 P_{k+1}} \, \mathbb{E} \left[ \left( X_k - \mathbb{E}[X_k \mid I^\mathcal{C}_{k}] \right)^2 \,\middle|\, I^\mathcal{C}_{k} \right]+ P_{k+1} \sigma_W^2 \\
    & \qquad + \sum_{j = k+1}^{N-1} \left( \frac{(ab \, P_{j+1})^2}{r + b^2 P_{j+1}} \mathbb{E} \left[ \left( X_j - \mathbb{E}[X_j \mid I^\mathcal{C}_j] \right)^2 \,\middle|\, I^\mathcal{C}_{k} \right] + P_{j+1} \sigma_W^2 \right) \\
    & = P_k \, \mathbb{E}[X_k^2 \mid I^\mathcal{C}_{k}] + \sum_{j = k}^{N-1} \left( \frac{(ab \, P_{j+1})^2}{r + b^2 P_{j+1}} \mathbb{E} \left[ \left( X_j - \mathbb{E}[X_j \mid I^\mathcal{C}_j] \right)^2 \,\middle|\, I^\mathcal{C}_{k} \right] + P_{j+1} \sigma_W^2 \right)
\end{align*}
where, in Step (a), we used the fact that
\begin{equation*}
    \mathbb{E} \left[ \left( \mathbb{E}[X_k \mid I^\mathcal{C}_{k}] \right)^2 \,\middle|\, I^\mathcal{C}_{k} \right] = \mathbb{E} \left[ 2 X_k \mathbb{E}[X_k \mid I^\mathcal{C}_{k}] - \left( \mathbb{E}[X_k \mid I^\mathcal{C}_{k}] \right)^2 \,\middle|\, I^\mathcal{C}_{k} \right].
\end{equation*}
Hence, the induction step is proved.

Note that the controller and cost-to-go functions are both dependent on the scheduling policy. Since we have not found the optimal scheduling policy, the controller is not fully determined, either. The analysis shows that the optimal controller is a linear function of the estimator. Hence, the effect of the delay can only be seen through $\mathbb{E}[X_k \mid I^\mathcal{C}_{k}]$ in $u_k^\star$. Secondly, in the steady-state, we know that
\[
\lim_{k \to \infty} P_k = P
\]
where \( P \) is the solution of
\[
P = q + a^2 P - \frac{(abP)^2}{r + b^2 P}.
\]
when $q, r > 0$ and for a scalar system \cite[Proposition 4.1]{bertsekas2012dynamic}.

Hence, for any $\tau$, the optimal controller is \[ u_k^\star = -L \; \mathbb{E}[X_k \mid I^\mathcal{C}_{k}]\] where $L = \frac{abP}{r + b^2 P}$.

\section{Proof of Theorem \ref{theorem:equivalent problem}} \label{Appendix_B}
The proof follows mostly from \cite[Chapter 8, Lemma 6.1]{aastrom2012introduction}. However, noisy observations have been considered in~\cite{aastrom2012introduction} without a scheduling mechanism and delay. We will show that the proof is valid for both intermittent observations and delay.

First of all, let's rewrite the objective as
\begin{equation}
    \frac{1}{N} \, \mathbb{E} \left[ q X_N^2 + \sum_{k=0}^{N-1} \left( q X_k^2 + r U_k^2 \right) \right] \label{eqn:Cost}
\end{equation}
and let
\begin{align*}
    P_k & = q + a^2 P_{k+1} - \frac{(ab \, P_{k+1})^2}{r + b^2 P_{k+1}} \\
    P_N & = q \\
    L_k & = \frac{ab P_{k+1}}{r + b^2 P_{k+1}}
\end{align*}
Now, let us define a telescoping operation
\[
\sum_{k=0}^{N-1} \left( P_{k+1} X_{k+1}^2 - P_k X_k^2 \right) = P_N X_N^2 - P_0 X_0^2.
\]
So,
\begin{align*}
    q X_N^2 & = P_0 X_0^2 + \sum_{k=0}^{N-1} \left( P_{k+1} X_{k+1}^2 - P_k X_k^2 \right) \\
    & = P_0 X_0^2 + \sum_{k=0}^{N-1} \left( P_{k+1} (a X_k + b U_k)^2 + 2 P_{k+1} (a X_k + b U_k) W_k + P_{k+1} W_k^2 - P_k X_k^2 \right) \numberthis \label{eqn:X_N_1}
\end{align*}
Note that
\begin{align*}
    \sum_{k=0}^{N-1} & \Big( P_{k+1}(aX_k + bU_k)^2 - P_k X_k^2 \Big) \\
    & \overset{(a)}{=} \sum_{k=0}^{N-1} \left( P_{k+1} a^2 X_k^2 + 2 P_{k+1} a b X_k U_k + P_{k+1} b^2 U_k^2 - \left( q + a^2 P_{k+1} - L_k^2 (r + b^2 P_{k+1}) \right) X_k^2 \right) \\
    & = \sum_{k=0}^{N-1} \left( (r + b^2 P_{k+1}) U_k^2 + (r + b^2 P_{k+1}) L_k^2 X_k^2 + 2 P_{k+1} a b X_k U_k \right)
    - \sum_{k=0}^{N-1} \left( q X_k^2 + r U_k^2 \right) \\
    & \overset{(b)}{=} \sum_{k=0}^{N-1} \left( (r + b^2 P_{k+1}) U_k^2 + (r + b^2 P_{k+1}) L_k^2 X_k^2 + 2 (r + b^2 P_{k+1}) L_k X_k U_k \right)
    - \sum_{k=0}^{N-1} \left( q X_k^2 + r U_k^2 \right) \\
    & = \sum_{k=0}^{N-1} \left( r + b^2 P_{k+1} \right) \left( L_k X_k + U_k \right)^2
- \sum_{k=0}^{N-1} \left( q X_k^2 + r U_k^2 \right)
\end{align*}
where in Step (a), we used the fact that
\begin{equation*}
    P_k = q + a^2 P_{k+1} - L_k^2 \left( r + b^2 P_{k+1} \right)
\end{equation*}
and in Step (b)
\begin{equation*}
    \left( r + b^2 P_{k+1} \right) L_k X_k U_k = P_{k+1} a b X_k U_k.
\end{equation*}
Hence, \eqref{eqn:X_N_1} becomes
\begin{align*}
    q X_N^2 + \sum_{k=0}^{N-1} \left( q X_k^2 + r U_k^2 \right)
    &= P_0 X_0^2 + \sum_{k=0}^{N-1} \left( r + b^2 P_{k+1} \right) \left( L_k X_k + U_k \right)^2 \\
    &\quad + \sum_{k=0}^{N-1} \left( 2 P_{k+1} (a X_k + b U_k) W_k + P_{k+1} W_k^2 \right). \numberthis \label{eqn:X_N_2}
\end{align*}
We use \eqref{eqn:X_N_2} in \eqref{eqn:Cost}:
\begin{align*}
    \frac{1}{N} \mathbb{E} \Bigg[ q X_N^2 + \sum_{k=0}^{N-1} \Big( q X_k^2 & + r U_k^2 \Big) \Bigg]
    = \frac{1}{N} \mathbb{E} \left[ P_0 X_0^2 + \sum_{k=0}^{N-1} \left( r + b^2 P_{k+1} \right) \left( L_k X_k + U_k \right)^2 \right. \\
    & \quad \left. + \sum_{k=0}^{N-1} \left( 2 P_{k+1} (a X_k + b U_k) W_k + P_{k+1} W_k^2 \right) \right] \\
    & = \frac{1}{N} \left(P_0 \, \mathbb{E} [X_0^2] + \mathbb{E} \left[ \sum_{k=0}^{N-1} \left( r + b^2 P_{k+1} \right) \left( L_k X_k + U_k \right)^2 \right] \right.\\
    & \quad \left. + \sum_{k=0}^{N-1} \left( 2 P_{k+1} \mathbb{E} \left[ (a X_k + b U_k) W_k \right] + P_{k+1} \mathbb{E} \left[ W_k^2 \right] \right) \right) \\
    & \overset{(a)}{=} \frac{1}{N} \left( P_0 \, \mathbb{E} [X_0^2] +\mathbb{E} \left[ \sum_{k=0}^{N-1} \left( r + b^2 P_{k+1} \right) \left( L_k X_k + U_k \right)^2 \right] + \sum_{k=0}^{N-1} P_{k+1} \, \sigma_W^2 \right) \\
    & \overset{(b)}{\approx}  \frac{1}{N} \left( P_0 \, \mathbb{E} [X_0^2] + \mathbb{E}\left[\sum_{k=0}^{N-1} L^2 \left( r + b^2 P \right) \left( X_k - \mathbb{E}[X_k | I^\mathcal{C}_{k}] \right)^2 \right] + N P \, \sigma_W^2 \right) \\
    & =  \frac{P_0}{N} \, \mathbb{E} [X_0^2] + \frac{1}{N} \mathbb{E}\left[\sum_{k=0}^{N-1} L^2 \left( r + b^2 P \right) \left( X_k - \mathbb{E}[X_k | I^\mathcal{C}_{k}] \right)^2\right] + P \, \sigma_W^2, \numberthis \label{eqn:eqv_prob_last_eqn}
\end{align*}
where the step (a) is because the expectations are unconditional and $W_k$ is independent of $X_k$ and $U_k$ for any $\tau$, and (b) is by using the optimal steady-state controller. Note that approximation (b) becomes accurate for sufficiently long horizons. Since $X_0$ is the initial state and is independent of $D_0$, equivalent problem $P_2$ follows from~\eqref{eqn:eqv_prob_last_eqn}.

\section{Proof of Lemma \ref{lemma:DP_Right_Boundry}} \label{Appendix_D}
From~\eqref{eqn:V_N-tau_Q1}, we observe that the value function depends only on \(\sigma_W^2\).
We proceed by induction.
Let \(\ell = N-\tau-\bar{s}\) for some \(\bar{s}=s+1 \ge 0\), and assume that~\eqref{eqn:V_l_Q_s+1} holds for the next stage.
Then,
\begin{align}
    V_{\ell}(I_{\ell}^{\mathcal S})\big|_{\,Q_{\ell}=\bar{s}}
    &=
    V'_{\ell\,\bar{s}\,1}(I_{\ell}^{\mathcal S}) \notag\\
    &=
    \sum_{j=0}^{\tau-1} a^{2(\tau-1-j)} \sigma_W^2
    + s
    \left(
        \sum_{j=0}^{\tau-1} a^{2(\tau-1-j)} \sigma_W^2
    \right) \notag\\
    &=
    \bar{s}
    \left(
        \sum_{j=0}^{\tau-1} a^{2(\tau-1-j)} \sigma_W^2
    \right).
\end{align}
Here, step~(a) follows from the DP recursion~\eqref{eqn:DP_rec} and the fact that, at stage \(k=\ell\) with budget \(Q_k=\bar{s}\), the scheduler can transmit at all subsequent time instants. Consequently, the minimizing action is \(D_k=1\), and only the corresponding stage cost is incurred.

Similarly, to establish~\eqref{eqn:V_l-1_Q_s+1}, observe that at stage \(\ell\), if the scheduler chooses not to transmit, i.e., \(D_{\ell}=0\), and the remaining budget satisfies \(Q_{\ell}=s-1\), then scheduling can be performed at all subsequent time instants. Consequently,~\eqref{eqn:V_l-1_Q_s+1} follows directly from the DP recursion~\eqref{eqn:DP_rec} and the next-stage value function under scheduling given by~\eqref{eqn:V_l_Q_s+1}.

\section{Proof of Theorem \ref{thm:sym_policies_are_optimal}} \label{Appendix_E}
In this proof, we focus on the value function at stage \(\ell\) with budget \(Q_{\ell}=s-1\), as defined in~\eqref{eqn:V_l-1_Q_s+1}. Our objective is to characterize the set of policies that minimize the corresponding expected cost, given by
\begin{equation}
    \min_{f_\ell(I_\ell^{\mathcal S})} \; \mathbb{E}\!\left[ V'_{\ell (s-1) 0}(I_\ell^{\mathcal S}) \right].
\end{equation}
Since this is an optimal control problem, optimality must hold at every stage of the decision process~\cite{bertsekas2012dynamic}.

Note that, since we evaluate the cost at stage \(\ell = N-\tau-s\) under the action \(D_\ell=0\), we can rewrite the objective by using total probability as
\begin{align}
    \mathbb{E}\!\left[ V'_{\ell (s-1) 0}(I_\ell^{\mathcal S}) \right]
    & =
    \sum_{m=1}^{N-Q_0-\tau-1}
    \Pr\!\left(D_{\ell-m-1}=1,\; Q_{\ell-m}=s-1\right) \notag\\
    & \qquad \qquad \mathbb{E}\!\left[
        V'_{\ell (s-1) 0}(I_\ell^{\mathcal S})
        \mid
        D_{\ell-m-1}=1,\; Q_{\ell-m}=s-1
    \right]
    \label{eqn:Exp_Vk_s+1}
\end{align}
where the upper bound follows directly from the geometry (see, for example, Fig.~\ref{fig:dp_diagram}).

We now focus on the conditional expectation term
\[
\mathbb{E}\!\left[ V'_{\ell(s-1)0}(I_\ell^{\mathcal S}) \,\middle|\, D_{\ell-m-1}=1,\; Q_{\ell-m}=s-1 \right].
\]
Using the state evolution in~\eqref{eqn:X_tk+m} together with~\eqref{eqn:V_l-1_Q_s+1}, we obtain
\begin{align}
    &\mathbb{E}\!\left[ V'_{\ell(s-1)0}(I_\ell^{\mathcal S}) \,\middle|\, D_{\ell-m-1}=1,\; Q_{\ell-m}=s-1 \right] \notag\\
    &\qquad =
    \mathbb{E}\!\left[
        a^{2\tau}\bar{\mathcal E}_\ell^{\,2}
        + s\!\left(\sum_{j=0}^{\tau-1} a^{2(\tau-1-j)}\sigma_W^2\right)
        \,\middle|\,
        D_{\ell-m-1}=1,\; Q_{\ell-m}=s-1
    \right] \notag\\
    &\qquad \overset{(a)}{=}
    \mathbb{E}\!\Bigg[
        a^{2\tau}\!\left(
            S_m
            - \mathbb{E}\!\left[
                S_m \,\middle|\,
                X_{\ell-m},\;
                D_{\ell-m-1}=1,\;
                D_{\ell-m}=0,\;
                \cdots, D_\ell=0
            \right]
        \right)^{\!2} \\
        & \qquad \qquad + s\!\left(\sum_{j=0}^{\tau-1} a^{2(\tau-1-j)}\sigma_W^2\right)
        \,\Bigg|\,
        D_{\ell-m-1}=1,\; Q_{\ell-m}=s-1
    \Bigg].
    \label{eqn:condExp_Vks0}
\end{align}
Here, step~(a) follows from the definition of \(\bar{\mathcal E}_k\) in~\eqref{eqn:E_bar} and the fact that, after the most recent update at time \(\ell-m\), the state process is Markov.

Hence, the problem reduces to
\begin{align}
    & \min_{f_\ell(I_\ell^{\mathcal S})} \;
    \mathbb{E}\!\left[ V'_{\ell(s-1)0}(I_\ell^{\mathcal S}) \right] \notag \\
    & =
    \min_{f_\ell(I_\ell^{\mathcal S})} \; \Bigg\{\sum_{m=1}^{N-Q_0-\tau-1}
    \Pr\!\left(D_{\ell-m-1}=1,\; Q_{\ell-m}=s-1\right)
    \mathbb{E}\!\left[
        V'_{\ell(s-1)0}(I_\ell^{\mathcal S})
        \,\middle|\,
        D_{\ell-m-1}=1,\; Q_{\ell-m}=s-1
    \right] \Bigg\} \notag\\
    & = \min_{f_\ell(I_\ell^{\mathcal S})} \; \Bigg\{
    \sum_{m=1}^{N-Q_0-\tau-1}
    \Pr\!\left(D_{\ell-m-1}=1,\; Q_{\ell-m}=s-1\right)
    \; \mathbb{E}\!\Bigg[
        s\!\left(
            \sum_{j=0}^{\tau-1} a^{2(\tau-1-j)} \sigma_W^2
        \right) \notag\\
    & \qquad 
        +\;
        a^{2\tau}
        \left(
            S_m
            - \mathbb{E}\!\left[
                S_m
                \,\middle|\,
                X_{\ell-m},\;
                D_{\ell-m-1}=1,\;
                D_{\ell-m}=\cdots=D_\ell=0
            \right]
        \right)^{\!2} \notag \\
        & \qquad \qquad \qquad \,\Bigg|\,
        D_{\ell-m-1}=1,\; Q_{\ell-m}=s-1
    \Bigg] \Bigg\} \label{eqn:Eqv_Exp_V_2}\\
    & \overset{(a)}{=} \min_{f_\ell(I_\ell^{\mathcal S})} \; \Bigg\{
    a^{2\tau} \sum_{m=1}^{N-Q_0-\tau-1}
    \Pr\!\left(D_{\ell-m-1}=1,\; Q_{\ell-m}=s-1\right)  \notag \\ 
    & \qquad \qquad \mathbb{E}\!\Bigg[
        \left( 
            S_m
            - C_m(X_{\ell-m},\pi)
        \right)^{\!2} \Bigg|\,
        D_{\ell-m-1}=1,\; Q_{\ell-m}=s-1
    \Bigg] \Bigg\}.
    \label{eqn:Eqv_Exp_V}
\end{align}
where in step (a) we used the fact that the first term is independent of $m$, and we defined 
\begin{equation}
    C_m(X_{\ell-m},\pi) \triangleq \mathbb{E}\!\left[
                S_m
                \,\middle|\,
                X_{\ell-m},\;
                D_{\ell-m-1}=1,\;
                D_{\ell-m}=\cdots=D_\ell=0
            \right]
\end{equation}
Now, we evaluate the expectation term in~\eqref{eqn:Eqv_Exp_V}:
\begin{align}
&\mathbb{E}\!\Big[
    (S_m - C_m(X_{\ell-m},\pi))^{2}
    \,\Big|\,
    D_{\ell-m-1}=1,\; Q_{\ell-m}=s-1
\Big] \notag \\
& \qquad =
\int_{-\infty}^{\infty}
\mathbb{E}\!\Big[
    (S_m - C_m(x,\pi))^{2}
    \,\Big|\,
    X_{\ell-m}=x,\; D_{\ell-m-1}=1,\; Q_{\ell-m}=s-1
\Big]\,
f'_X(x)\,dx .
\end{align}
where we defined \( f'_X(x) \triangleq f_{X_{\ell-m}\mid D_{\ell-m}=1,\;Q_{\ell-m}=s-1}(x). \)

Note that
\begin{align}
    &\mathbb{E}\!\Big[
        \big(S_m - C_m(X_{\ell-m},\pi)\big)^{2}
        \,\Big|\,
        X_{\ell-m}=x,\; D_{\ell-m-1}=1,\; Q_{\ell-m}=s-1
    \Big] \notag\\
    &\qquad \overset{(a)}{=}
    \mathbb{E}\!\Big[
        S_m^{2}
        \,\Big|\,
        X_{\ell-m}=x,\; D_{\ell-m-1}=1,\; Q_{\ell-m}=s-1
    \Big]
    + \big(C_m(x,\pi)\big)^{2} \notag\\
    &\qquad\qquad
    -2\,C_m(X_{\ell-m},\pi)\,
    \mathbb{E}\!\Big[
        S_m
        \,\Big|\,
        X_{\ell-m}=x,\; D_{\ell-m-1}=1,\; Q_{\ell-m}=s-1
    \Big] \notag\\
    &\qquad \overset{(b)}{=}
    \mathbb{E}\!\Big[ S_m^2 + \big(C_m(X_{\ell-m},\pi)\big)^{2}\,\Big|\,
        X_{\ell-m}=x,\; D_{\ell-m-1}=1,\; Q_{\ell-m}=s-1
    \Big].
    \label{eqn:expand_quad}
\end{align}
In step~(a), we used the fact that \(C_m(x,\pi)\) is deterministic under the conditioning \(X_{\ell-m}=x\).
In step~(b), we used the causality of the scheduling policy together with the independence and zero-mean property of \(\{W_k\}\). Since \(S_m=\sum_{j=0}^{m-1} a^{m-1-j} W_{\ell-m+j}\) depends only on the future disturbances
\(\{W_{\ell-m},\dots,W_{\ell-1}\}\), while \((X_{\ell-m},D_{\ell-m},Q_{\ell-m})\) is measurable with respect to
\(\sigma(W_0,\dots,W_{\ell-m-1})\), it follows that
\[
\mathbb{E}\!\left[S_m C_m(X,\pi)\mid X_{\ell-m}=x,\; D_{\ell-m-1}=1,\; Q_{\ell-m}=s-1\right]=0,
\]
and
\[
\mathbb{E}\!\left[S_m^2 \mid X_{\ell-m}=x,\; D_{\ell-m-1}=1,\; Q_{\ell-m}=s-1\right]
= \sigma_W^2 \sum_{j=0}^{m-1} a^{2(m-1-j)} .
\]
Therefore, the problem in~\eqref{eqn:Eqv_Exp_V} can be written as
\begin{align}
    & \min_{f_\ell(I_\ell^{\mathcal S})}
    \Bigg\{
        a^{2\tau}
        \sum_{m=1}^{N-Q_0-\tau-1}
        \Pr\!\left(D_{\ell-m-1}=1,\; Q_{\ell-m}=s-1\right) \notag \\
        & \qquad \mathbb{E}\!\Bigg[
            \big(
                S_m - C_m(X_{\ell-m},\pi)
            \big)^{2}
            \,\Big|\,
            D_{\ell-m-1}=1,\; Q_{\ell-m}=s-1
        \Bigg]
    \Bigg\} \notag
\end{align}
Note that
\begin{align}
	&\mathbb{E}\!\Big[
	(S_m-C_m(X_{\ell-m},\pi))^2
	\Big| X_{\ell-m}=x,\,  D_{\ell-m-1}=1,\,
	Q_{\ell-m}=s-1
	\Big] \notag\\
	& \qquad \qquad \overset{(a)}{=}
	\mathbb{E}\!\left[
	S_m^2
	\Big| X_{\ell-m}=x,\,
	D_{\ell-m-1}=1,\,
	Q_{\ell-m}=s-1
	\right] \notag \\
	&  \qquad \qquad \qquad + \big(C_m(x,\pi)\big)^2  -2C_m(x,\pi)
	\mathbb{E}\!\Big[
	S_m
	\Big| X_{\ell-m}=x,\, D_{\ell-m-1}=1,\,
	Q_{\ell-m}=s-1
	\Big] \notag\\
	& \qquad \qquad \overset{(b)}{=}
	\sigma_W^2 \sum_{j=0}^{m-1} a^{2(m-1-j)}
	+
	\big(C_m(x,\pi)\big)^2 .
	\label{eqn:Cm_expansion}
\end{align}
In step (a), we used the fact that \(C_m(x,\pi)\) is deterministic under the conditioning \(X_{\ell-m}=x\). In step (b), we used the causality of the scheduling policy together with the independence and zero-mean property of \(\{W_k\}\). Since
\(
S_m=\sum_{j=0}^{m-1}a^{m-1-j}W_{\ell-m+j}
\)
depends only on the future disturbances \(\{W_{\ell-m},\dots,W_{\ell-1}\}\), while \((X_{\ell-m},D_{\ell-m},Q_{\ell-m})\) is measurable with respect to \(\sigma(W_0,\dots,W_{\ell-m-1})\), it follows that
\begin{equation*}
	\mathbb{E}\!\left[
	S_m
	\mid X_{\ell-m}=x,\,
	D_{\ell-m-1}=1,\,
	Q_{\ell-m}=s-1
	\right]=0.
\end{equation*}
For each fixed past branch, the variance term
\(
\sigma_W^2 \sum_{j=0}^{m-1} a^{2(m-1-j)}
\)
in~\eqref{eqn:Cm_expansion} is fixed and independent of the current scheduling rule \(f_\ell(I_\ell^{\mathcal S})\). Hence, it does not affect the minimization over \(f_\ell(I_\ell^{\mathcal S})\). Therefore, up to branchwise additive constants, the minimization in~\eqref{eqn:Eqv_Exp_V} reduces to
\begin{align}
	& \min_{f_\ell(I_\ell^{\mathcal S})}
	\Bigg\{
	a^{2\tau}
	\sum_{m=1}^{N-Q_0-\tau-1}
	\Pr(D_{\ell-m-1}=1,Q_{\ell-m}=s-1) \notag \\
	& \qquad \qquad \int_{-\infty}^{\infty}
	\big(C_m(x,\pi)\big)^2 f'_X(x)\,dx
	\Bigg\}.
	\label{eqn:min_Cm}
\end{align}
Since \(\big(C_m(x,\pi)\big)^2 \ge 0\) and \(f'_X(x)\ge 0\), the objective in~\eqref{eqn:min_Cm} is minimized when
\begin{equation*}
	C_m(X_{\ell-m},\pi)=0
	\; \text{almost surely},
	\;
	m\in\{1,\dots,N-\tau-1\}.
\end{equation*}
This is precisely the defining property of symmetric policies in Definition~\eqref{defn:sym_pol}. Hence, symmetric policies eliminate the conditional-bias term and attain the branchwise minimum. Since the argument holds for every past branch in the dynamic program, symmetric policies are optimal for the constrained LQR problem \(P_2\).

\section{Proof of Lemma~\ref{lemma:est_error_evolution}} \label{Appendix_F}
\begin{align}
\mathbb{E}\!\left[ X_{t_k+m+1} \mid I^{\mathcal C}_{t_k+m+1} \right]
&= \mathbb{E}\!\left[
a X_{t_k+m} + b U_{t_k+m} + W_{t_k+m}
\;\middle|\; I^{\mathcal C}_{t_k+m+1}
\right] \notag\\
&\overset{(a)}{=} a\,\mathbb{E}\!\left[
X_{t_k+m} \mid I^{\mathcal C}_{t_k+m+1}
\right]
+ b\,U_{t_k+m}\notag\\
&= a\,\mathbb{E}\!\left[
X_{t_k+m} \mid I^{\mathcal C}_{t_k+m+1}
\right]
- bL\,\mathbb{E}\!\left[
X_{t_k+m} \mid I^{\mathcal C}_{t_k+m}
\right] \notag\\
&= (a-bL)\,\mathbb{E}\!\left[
X_{t_k+m} \mid I^{\mathcal C}_{t_k+m}
\right] \notag \\
& \qquad + a\!\left(
\mathbb{E}\!\left[
X_{t_k+m} \mid I^{\mathcal C}_{t_k+m+1}
\right]
- \mathbb{E}\!\left[
X_{t_k+m} \mid I^{\mathcal C}_{t_k+m}
\right]
\right).
\label{eqn:Exp_X_tk+m+1}
\end{align}
where in step~(a), we used the fact that, due to the delay~$\tau \ge 1$, the controller information
\(I^{\mathcal C}_{t_k+m+1}\) does not depend on the disturbance \(W_{t_k+m}\), and hence
\(\mathbb{E}[W_{t_k+m}\mid I^{\mathcal C}_{t_k+m+1}]=0\).

Note that for all \(m\),
\begin{equation}
X_{t_k+m}
=
a^{m} X_{t_k}
+ \sum_{j=0}^{m-1} a^{m-1-j} b\,U_{t_k+j}
+ \sum_{j=0}^{m-1} a^{m-1-j} W_{t_k+j}.
\end{equation}

Taking conditional expectations with respect to \(I^{\mathcal C}_{t_k+m}\), we obtain
\begin{equation}
\mathbb{E}\!\left[ X_{t_k+m} \mid I^{\mathcal C}_{t_k+m} \right]
=
a^{m} X_{t_k}
+ \sum_{j=0}^{m-1} a^{m-1-j} b\,U_{t_k+j}
+ \sum_{j=0}^{m-1} a^{m-1-j}
\mathbb{E}\!\left[ W_{t_k+j} \mid I^{\mathcal C}_{t_k+m} \right].
\end{equation}

Similarly,
\begin{equation}
\mathbb{E}\!\left[ X_{t_k+m} \mid I^{\mathcal C}_{t_k+m+1} \right]
=
a^{m} X_{t_k}
+ \sum_{j=0}^{m-1} a^{m-1-j} b\,U_{t_k+j}
+ \sum_{j=0}^{m-1} a^{m-1-j}
\mathbb{E}\!\left[ W_{t_k+j} \mid I^{\mathcal C}_{t_k+m+1} \right].
\end{equation}

Therefore,
\begin{align}
&\mathbb{E}\!\left[ X_{t_k+m} \mid I^{\mathcal C}_{t_k+m+1} \right]
- \mathbb{E}\!\left[ X_{t_k+m} \mid I^{\mathcal C}_{t_k+m} \right] \notag\\
&\qquad =
\sum_{j=0}^{m-1} a^{m-1-j}
\left(
\mathbb{E}\!\left[ W_{t_k+j} \mid I^{\mathcal C}_{t_k+m+1} \right]
-
\mathbb{E}\!\left[ W_{t_k+j} \mid I^{\mathcal C}_{t_k+m} \right]
\right).
\label{eqn:cond_exp_difference}
\end{align}

Hence, \eqref{eqn:Exp_X_tk+m+1} becomes
\begin{align}
\mathbb{E}\!\left[ X_{t_k+m+1} \mid I^{\mathcal C}_{t_k+m+1} \right]
&=  (a-bL)\,\mathbb{E}\!\left[ X_{t_k+m} \mid I^{\mathcal C}_{t_k+m} \right] \notag\\
&\quad
+ a \sum_{j=0}^{m-1} a^{m-1-j}
\left(
\mathbb{E}\!\left[ W_{t_k+j} \mid I^{\mathcal C}_{t_k+m+1} \right]
-
\mathbb{E}\!\left[ W_{t_k+j} \mid I^{\mathcal C}_{t_k+m} \right]
\right).
\label{eqn:4}
\end{align}

Note that
\begin{align}
a \sum_{j=0}^{m-1} a^{m-1-j}
\mathbb{E}\!\left[ W_{t_k+j} \mid I^{\mathcal C}_{t_k+m+1} \right]
&= \sum_{j=0}^{m-1} a^{m-j}
\mathbb{E}\!\left[ W_{t_k+j} \mid I^{\mathcal C}_{t_k+m+1} \right] \notag\\
&= \sum_{j=0}^{m} a^{m-j}
\mathbb{E}\!\left[ W_{t_k+j} \mid I^{\mathcal C}_{t_k+m+1} \right]
- \mathbb{E}\!\left[ W_{t_k+m} \mid I^{\mathcal C}_{t_k+m+1} \right] \notag\\
&= \sum_{j=0}^{m} a^{m-j}
\mathbb{E}\!\left[ W_{t_k+j} \mid I^{\mathcal C}_{t_k+m+1} \right] \notag\\
&= \mathbb{E}\!\left[ S_{m+1} \mid I^{\mathcal C}_{t_k+m+1} \right].
\label{eqn:Sm_plus_1_relation}
\end{align}
The third equality follows since, for \(m \ge \tau\), the controller information
\(I^{\mathcal C}_{t_k+m+1}\) depends on scheduling decisions only up to
\(D_{t_k+m+1-\tau}\) and past disturbances, and is therefore independent of the
current disturbance \(W_{t_k+m}\). Hence,
\(\mathbb{E}[W_{t_k+m}\mid I^{\mathcal C}_{t_k+m+1}]=0\).

Hence,
\begin{equation}
\mathbb{E}\!\left[ X_{t_k+m+1} \mid I^{\mathcal C}_{t_k+m+1} \right]
=
(a-bL)\,\mathbb{E}\!\left[ X_{t_k+m} \mid I^{\mathcal C}_{t_k+m} \right]
+ \mathbb{E}\!\left[ S_{m+1} \mid I^{\mathcal C}_{t_k+m+1} \right]
- a\,\mathbb{E}\!\left[ S_m \mid I^{\mathcal C}_{t_k+m} \right].
\end{equation}

\section{State-Based Policy}
\label{appendix:G}

We derive the controller corresponding to the state-based scheduling policy
\[
D_k = \mathds{1}_{\{|X_k| > \gamma\}}.
\]

Define
\[
\xi_1 \triangleq a X_{t_k} - bL \hat{X}_{t_k}.
\]
Then
\[
X_{t_k+1} = \xi_1 + W_{t_k}.
\]

Conditioning on $|X_{t_k+1}|<\gamma$,
\[
\hat{X}_{t_k+1}
=
\xi_1 + C_{0,1},
\qquad
C_{0,1}
=
\mathbb{E}\!\left[
W_{t_k}\mid |\xi_1+W_{t_k}|<\gamma
\right].
\]

At the next step,
\begin{align}
X_{t_k+2}
&=
a X_{t_k+1}
-
bL \hat{X}_{t_k+1}
+
W_{t_k+1} \notag\\
&=
(a-bL)\xi_1
-
bL C_{0,1}
+
a W_{t_k}
+
W_{t_k+1}.
\end{align}

Define
\[
\xi_2
\triangleq
(a-bL)\xi_1
-
bL C_{0,1}.
\]
Then
\[
X_{t_k+2}
=
\xi_2
+
a W_{t_k}
+
W_{t_k+1}.
\]

Proceeding recursively, for $m \ge 1$,
\[
X_{t_k+m}
=
\xi_m
+
\sum_{i=0}^{m-1}
a^{m-1-i} W_{t_k+i},
\]
where for $m \ge 2$
\begin{equation}
\xi_m
=
(a-bL)\xi_{m-1}
-
bL
\sum_{i=0}^{m-2}
a^{m-2-i} C_{i,m-1}.
\label{eq:xi_recursion}
\end{equation}

Define the conditioning event
\[
\mathcal{A}_m
=
\bigcap_{j=1}^m
\left\{
\left|
\xi_j
+
\sum_{i=0}^{j-1}
a^{j-1-i} W_{t_k+i}
\right|
\le \gamma
\right\}.
\]

Then the conditional state estimate is
\begin{equation}
\hat{X}_{t_k+m}
=
\xi_m
+
\sum_{i=0}^{m-1}
a^{m-1-i} C_{i,m},
\qquad
C_{i,m}
=
\mathbb{E}[W_{t_k+i}\mid \mathcal{A}_m].
\label{eq:state_est_general}
\end{equation}

Finally, the resulting controller becomes
\begin{equation}
U^{\mathrm{State}}_{t_k+m}
=
-L
\left(
\xi_m
+
\sum_{i=0}^{m-1}
a^{m-1-i} C_{i,m}
\right),
\end{equation}
which matches~\eqref{eqn:State_controller}.

\section{Extension to the Transient Case}
\label{appendix:E}
In the transient case, the optimal controller in Thm.~\ref{theorem:optimal control} becomes
\begin{equation}
    U_k = \gamma_k(I^\mathcal{C}_{k}) = -L_k\,\hat{X}_k,
    \label{eqn:optimal_controller_transient}
\end{equation}
where
\begin{equation}
    \hat{X}_k = \mathbb{E}\bigl[X_k \mid I^\mathcal{C}_{k}\bigr], 
    \quad
    L_k = \frac{a b\,P_{k+1}}{r + b^2 P_{k+1}},
    \quad
    P_k = a^2 P_{k+1} + q - \frac{(a b P_{k+1})^2}{r + b^2 P_{k+1}}.
\end{equation}
Hence, under this controller, the equivalent problem in Thm.~\ref{theorem:equivalent problem} becomes
\begin{align}
P'_{2}:\quad 
  &\min_{\boldsymbol{\pi} \in \Pi}\; \frac{1}{N}\,\mathbb{E}\left[\sum_{k=0}^{N-1} L_k \bigl(r + b^2 P_{k+1}\bigr)\mathcal{E}_{k}^{2}\right] 
  \label{eqn:P_2_transient}\\
\text{s.t.}\quad 
  & \frac{1}{N}\sum_{k=0}^{N-1} D_{k} \le r_{s}. 
  \nonumber
\end{align}
Compared to the steady-state case, the objective now includes an additional time-varying coefficient. This term propagates into the dynamic programming formulation and consequently appears in all subsequent equations.

First, the stage cost $g_k(I_k^{\mathcal S})$ in \eqref{eqn:g_k} is modified as follows:
\begin{align}
    g_k(I_k^{\mathcal S})
    = 
    \begin{cases}
        \displaystyle
        L_{k+\tau}(r + b^2 P_{k+\tau+1}) \sum_{j=0}^{\tau-1} a^{2(\tau-1-j)} \sigma_W^2,
        & \text{if } k = t_k, \\[1.2ex]
        \displaystyle
        L_{k+\tau}(r + b^2 P_{k+\tau+1}) \Bigg(
        a^{2\tau}
        \left(
            X_k
            - \mathbb{E}\!\left[ X_k \mid I_{k+\tau}^{\mathcal C} \right]
        \right)^2
        + \sum_{j=0}^{\tau-1} a^{2(\tau-1-j)} \sigma_W^2 \Bigg),
        & \text{otherwise.}
    \end{cases}
\end{align}
Accordingly, the DP recursion becomes
\begin{align}
    V_k(I_k^{\mathcal S})
    = \min_{D_k} \Bigg\{
    & L_{k+\tau}(r + b^2 P_{k+\tau+1}) a^{2\tau} \bar{\mathcal{E}}_k^2
      + \mathbb{E}\!\left[ V_{k+1}(I_{k+1}^{\mathcal S}) \mid I_k^{\mathcal S}, D_k = 0 \right], \notag\\
    & \mathbb{E}\!\left[ V_{k+1}(I_{k+1}^{\mathcal S}) \mid I_k^{\mathcal S}, D_k = 1 \right]
    \Bigg\}
    + L_{k+\tau}(r + b^2 P_{k+\tau+1}) \sum_{j=0}^{\tau-1} a^{2(\tau-1-j)} \sigma_W^2 .
    \label{eqn:DP_rec_transient}
\end{align}
The terminal-stage expressions in \eqref{eqn:V_N-tau_Q1} and \eqref{eqn:V_N-tau_Q0} become
\begin{equation}
    V'_{(N-\tau-1)\,1\,1}(I_{N-\tau-1}^{\mathcal S})
    = L_{N-1}(r + b^2 P_N) \sum_{j=0}^{\tau-1} a^{2(\tau-1-j)} \sigma_W^2 ,
\end{equation}
and
\begin{equation}
    V'_{(N-\tau-1)\,0\,0}(I_{N-\tau-1}^{\mathcal S})
    = L_{N-1}(r + b^2 P_N) \Big(
    a^{2\tau} \bar{\mathcal{E}}_{N-\tau-1}^{\,2}
    + \sum_{j=0}^{\tau-1} a^{2(\tau-1-j)} \sigma_W^2 \Big) ,
\end{equation}
respectively.

Hence, let $\ell = N - \tau - s$ for some $s \ge 1$. Then, Lemma~\ref{lemma:DP_Right_Boundry} becomes
\begin{equation}
    V'_{\ell\,s\,1}(I_{\ell}^{\mathcal S})
    = \left(\sum_{k=N-s}^{N-1} L_k (r + b^2 P_{k+1})\right)
    \left(
        \sum_{j=0}^{\tau-1} a^{2(\tau-1-j)} \sigma_W^2
    \right),
\end{equation}
and
\begin{equation}
    V'_{\ell\,(s-1)\,0}(I_{\ell}^{\mathcal S})
    = L_{N-s}(r + b^2 P_{N-s+1}) a^{2\tau} \bar{\mathcal{E}}_{\ell}^{\,2}
      + \left(\sum_{k=N-s}^{N-1} L_k (r + b^2 P_{k+1})\right)
      \left(
          \sum_{j=0}^{\tau-1} a^{2(\tau-1-j)} \sigma_W^2
      \right),
\end{equation}

We now proceed to the separation proof. Since Lemma~\ref{lemma:DP_Right_Boundry} has changed, the expression in \eqref{eqn:Eqv_Exp_V_2} is modified as follows:
\begin{align}
    & \min_{f_\ell(I_\ell^{\mathcal S})} \;
    \mathbb{E}\!\left[ V'_{\ell\,(s-1)\,0}(I_\ell^{\mathcal S}) \right] \notag \\
    & = \min_{f_\ell(I_\ell^{\mathcal S})} \; \Bigg\{
    \sum_{m=1}^{N - Q_0 - \tau - 1}
    \Pr\!\left(D_{\ell-m} = 1,\; Q_{\ell-m} = s-1\right)  \notag \\ 
    & \qquad \qquad \times \mathbb{E}\!\Bigg[
        L_{N-s}(r + b^2 P_{N-s+1}) a^{2\tau}
        \left( 
            S_m - C_m(X_{\ell-m}, \pi)
        \right)^{2} \,\Bigg|\,
        D_{\ell-m} = 1,\; Q_{\ell-m} = s-1
    \Bigg] \notag \\
    & \qquad \qquad + \left(\sum_{k=N-s}^{N-1} L_k (r + b^2 P_{k+1})\right)
    \bigg(\sum_{j=0}^{\tau-1} a^{2(\tau-1-j)} \sigma_W^2 \bigg)
    \Bigg\}.
\end{align}

Since the term
\[
\left(\sum_{k=N-s}^{N-1} L_k (r + b^2 P_{k+1})\right)
\bigg(\sum_{j=0}^{\tau-1} a^{2(\tau-1-j)} \sigma_W^2 \bigg)
\]
is independent of $m$, and since $L_{N-s}(r + b^2 P_{N-s+1}) > 0$, we recover the same expression as in \eqref{eqn:Eqv_Exp_V}. Therefore, symmetric policies remain optimal, and the separation principle continues to hold under the transient controller. We now proceed to the optimal scheduling mechanism.

With the modified DP recursion in \eqref{eqn:DP_rec_transient}, Lemma~\ref{lem:zero_budget_recursion} becomes
\begin{equation}
    V'_{(k-1)\,0\,0}\!\left(I_{k-1}^{\mathcal S}\right)
    =
    s_{(k-1)0}\,\bar{\mathcal{E}}_{k-1}^{2}
    +
    c_{(k-1)00}\,\sigma_W^{2},
    \label{eqn:V_k-1_00_transient}
\end{equation}
where
\begin{align}
    s_{(k-1)0} &\triangleq L_{k+\tau-1}(r + b^2 P_{k+\tau}) a^{2\tau} + a^{2} s_{k0}, \\
    c_{(k-1)00} &\triangleq s_{k0} + c_{k00} + L_{k+\tau-1}(r + b^2 P_{k+\tau}) \sum_{j=0}^{\tau-1} a^{2(\tau-1-j)} .
\end{align}

Similarly, Lemma~\ref{lemma:D_1_Q_1} becomes
\begin{equation}
   V'_{(k-1)\,1\,1}\!\left(I_{k-1}^{\mathcal S}\right)
    =
    c_{(k-1)11}\,\sigma_W^{2},
\end{equation}
where
\begin{equation}
    c_{(k-1)11} \triangleq s_{k0} + c_{k00} + L_{k+\tau-1}(r + b^2 P_{k+\tau})
    \left(\sum_{j=0}^{\tau-1} a^{2(\tau-1-j)} \right).
\end{equation}

Finally, Lemmas~\ref{lemma:Diagonal} and \ref{lemma:Horizontal} become
\begin{equation}
    V'_{k(j-1)1}\!\left(I^{\mathcal S}_{k}\right)
    =
    c_{k(j-1)1}\,\sigma_W^{2} + z_{k(j-1)1},
    \label{eqn:V_k_j-1_1_transient}
\end{equation}
where
\begin{align}
    c_{k(j-1)1}
    & \triangleq
    c_{(k+1)j0}\,
    \Pr\!\big(D_{k+1}=0 \mid D_k=1\big)
    + c_{(k+1)j1}\,
    \Pr\!\big(D_{k+1}=1 \mid D_k=1\big) \notag \\
    & \qquad +
    L_{k+\tau}(r + b^2 P_{k+\tau+1})
    \sum_{i=0}^{\tau-1} a^{2(\tau-1-i)}.
\end{align}
and
\begin{equation}
    V'_{(k-1)j0}\!\left(I^{\mathcal S}_{k-1}\right)
    =
    s_{(k-1)j}\,\bar{\mathcal E}_{k-1}^{2}
    +
    c_{(k-1)j0}\,\sigma_W^{2}
    + z_{(k-1)j0},
    \label{eqn:V_k-1_j0_transient}
\end{equation}
where
\begin{align}
    s_{(k-1)j}
    &\triangleq
    a^{2\tau} L_{k+\tau-1}(r + b^2 P_{k+\tau})
    + a^{2}\Pr\!\big(D_k=0 \mid D_{k-1}=0\big) s_{kj},
    \\[0.3em]
    c_{(k-1)j0}
    &\triangleq
    \Pr\!\big(D_k=0 \mid D_{k-1}=0\big)c_{k j 0}
    +
    \Pr\!\big(D_k=1 \mid D_{k-1}=0\big)c_{k j 1} \notag \\
    & \qquad +
    L_{k+\tau-1}(r + b^2 P_{k+\tau})
    \sum_{i=0}^{\tau-1} a^{2(\tau-1-i)} .
\end{align}
The terms $z_{k(j-1)1}$ and $z_{(k-1)j0}$ remain the same as in Lemmas~\ref{lemma:Diagonal} and \ref{lemma:Horizontal}, respectively.

\bibliographystyle{IEEEtran}
\bibliography{references_V2.bib}

\end{document}